\documentclass[11pt, oneside]{article}   	% use "amsart" instead of "article" for AMSLaTeX format
\usepackage{geometry}                		% See geometry.pdf to learn the layout options. There are lots.
\usepackage[dvipdfmx]{graphicx}
\usepackage[usenames]{xcolor}		
\usepackage{amsmath,amsthm,amssymb}
\usepackage{amscd}
\usepackage[all]{xy}
\usepackage{comment}

\theoremstyle{definition}
\newtheorem{dfn}{Definition}[section]

\theoremstyle{plain}
\newtheorem{thm}[dfn]{Theorem}
\newtheorem{prop}[dfn]{Proposition}
\newtheorem{lem}[dfn]{Lemma}

\newtheorem{rmk}[dfn]{Remark}

\newcommand{\K}{\mathbb{K}}
\newcommand{\Z}{\mathbb{Z}}

\newcommand{\der}{\mathrm{Der}}
\newcommand{\aut}{\mathrm{Aut}}
\newcommand{\tder}{\mathrm{tder}}

\newcommand{\tDer}{\mathrm{tDer}}

\renewcommand{\hom}{\mathrm{Hom}}
\newcommand{\taut}{\mathrm{TAut}}

\renewcommand{\div}{\mathrm{div}}

\newcommand{\gdiv}{\mathrm{gDiv}}

\newcommand{\gn}{{(g,n+1)}}

\newcommand{\pa}{\partial}

\newcommand{\gr}{\mathrm{gr}}
\newcommand{\ad}{\mathrm{ad}}

\newcommand{\rot}{\mathrm{rot}}

\newcommand{\grwt}{\mathrm{gr}^{\mathrm{wt}}}
\newcommand{\IntS}{\mathrm{Int}\,S}
\newcommand{\inv}{^{-1}}
\newcommand{\R}{\mathbb{R}}
\newcommand{\Rp}{{\mathbb{R}_{>0}}}
\newcommand{\Rnn}{{\mathbb{R}_{\geq 0}}}
\newcommand{\smooth}{C^\infty}
\newcommand{\Sph}{\mathbb{S}}
\newcommand{\Ss}{\mathcal{S}}
\newcommand{\wt}{\mathrm{wt}}

\title{Some algebraic aspects of the Turaev cobracket
}
\author{Nariya Kawazumi\thanks{Department of Mathematical Sciences, University of Tokyo, 3-8-1 Komaba, Meguro-ku, Tokyo 153-8914, Japan \texttt{e-mail:kawazumi@ms.u-tokyo.ac.jp}}}
\date{}							% Activate to display a given date or no date

\begin{document}

\maketitle

\begin{abstract}
The Turaev cobracket, a loop operation introduced by V.\ Turaev \cite{Tu91}, 
which measures the self-intersection of a loop on a surface, 
is a modification of a path operation introduced earlier 
by Turaev himself \cite{Tu78}, as well as 
a counterpart of the Goldman bracket \cite{Go86}.
In this survey based on the author's joint works with A.\ Alekseev, 
Y.\ Kuno and F.\ Naef, 
we review some algebraic aspects of the cobracket and its framed variants
including their formal description, 
an application to the mapping class group of the surface
and a relation to the (higher genus) Kashiwara--Vergne problem. 
In addition, we review a homological description of the cobracket 
after R.\ Hain \cite{Hain18}.
\end{abstract}

\tableofcontents

\section{Introduction}
\label{sec:intro}

Let $S$ be a connected oriented surface. 
By a {\it surface} we mean a $2$-dimensional $\smooth$ manifold
possibly with boundary.
The set $[S^1, S]$ of free homotopy classes of loops on $S$, 
which we denote by $\hat\pi = \hat\pi(S)$, 
is naturally identified with the set of conjugacy 
classes $\pi_1(S)/\text{conj}$ of the fundamental group $\pi_1(S)$. Goldman \cite{Go86} introduced 
a binary loop operation on $\hat\pi$, which we call
the Goldman bracket, on the free $\Z$-module $\Z\hat\pi$ over 
the set $\hat\pi$ as a topological extract
of the Poisson structure on the moduli space of 
flat bundles over the surface $S$. 
Earlier than the Goldman bracket, Turaev \cite{Tu78} introduced
a loop operation on the fundamental group $\pi_1(S)$.
The Turaev cobracket \cite{Tu91} is its free loop version 
as well as a counterpart of the Goldman bracket. 
The cobracket together with the Goldman bracket 
makes the quotient of $\Z\hat\pi$ by the linear span 
of the trivial loop a Lie bialgebra \cite{Tu91}.
We consider framed versions of the Turaev cobracket
which make $\Z\hat\pi$ itself a Lie bialgebra.
\par
Yusuke Kuno and the author have studied these loop operations 
and their applications to the mapping class group of a surface
\cite{KK14, KKpre, KK15, KK16}. There, the formal description of 
the Goldman bracket through a special expansion 
plays a fundamental role. 
The formal description of the framed Turaev cobracket was desirable
but in vain. In the genus $0$ case it was carried out 
by Massuyeau \cite{Mas15}. 
To obtain the complete realization of the formal description, 
it was essential for us to collaborate with Anton Alekseev and Florian Naef.
Here the key concept is the Kashiwara--Vergne problem
\cite{KV78, AT12}. 
Alekseev--Kawazumi--Kuno--Naef proved that
the set of special expansions which induce the formal description 
of the framed Turaev cobracket has a one-to-one correspondence 
with the set of solutions to the Kashiwara--Vergne problem.
More precisely, the Kashiwara--Vergne problem involves genus $0$ surfaces.
We first proved the correspondence for the genus $0$ case 
\cite{genus0}. This leads us to formulate a positive genus version of 
the Kashiwara--Vergne problem \cite{announce}. In \cite{highergenus}
we proved the correspondence for all positive genus surface.
Moreover we proved the existence of solutions to the positive genus version
except some genus $1$ cases, where there is no solution.
\par
The present paper sketches this story, and 
is a continuation of the author's joint survey with Kuno 
about the Goldman--Turaev Lie bialgebra and its applications 
to the mapping class group \cite{KK16}.
Independently, Hain is writing a survey on this topic \cite{Hain19}.
His approach is ``motivic", i.e. based on his mixed Hodge theory, 
while our approach comes from low-dimensional topology.
\par
In \S\ref{sec:def} we recall the definition of the Turaev cobracket 
and its framed variants. The Turaev cobracket is defined 
for a generic immersed loop on the surface. 
It is invariant under regular-homotopy, but not under homotopy.
In fact, the birth-death move $(\omega1)$ of a monogon changes 
the value of the Turaev cobracket. This is the reason why we have to take the 
quotient by the linear span of the trivial loop. To avoid taking 
the quotient, we consider a framed version of the cobracket.
We need the coefficient
of the trivial loop in the framed Turaev cobracket
to get the Enomoto--Satoh trace \cite{ES} and 
the Alekseev--Torossian divergence cocycle \cite{AT12}. \par
A straightforward computation can show 
the regular-homotopy invariance of the Turaev cobracket.
But this is proved also by a homological description of the cobracket.
Following Hain \cite{Hain18} we will give it in \S\ref{sec:homol}.
Hain used his own description to prove that the framed cobracket 
is a morphism of mixed Hodge structures. 
Combining this result and our theorems (Theorems \ref{thm:genus0} and \ref{thm:genus+}), he
deduces the existence of solutions to the Kashiwara--Vergne 
problem under the same condition as ours
\cite{Hain18}.
\par
\S\ref{sec:FG} has some overlap with our previous survey paper 
\cite{KK16}. 
We formulate our formality problems in \S\ref{subsec:exp}. 
\S\ref{subsec:gold} is devoted to a review on the formality of 
the Goldman bracket. An application to the mapping class group 
is mentioned in \S\ref{subsec:log}. 
A geometric version $\tau$ of the Johnson homomorphisms 
is introduced by the logarithm of mapping classes. 
We will give an alternative 
proof of the vanishing of the composite of the framed Turaev cobracket 
and the homomorphism $\tau$ on the Johnson group $\mathcal{K}_{g,1}$.
\par
We survey the results of 
\cite{announce, genus0, highergenus} in \S\ref{sec:form}.
%A formal description of the framed cobracket in genus $0$ 
%was given first by Massuyeau \cite{Mas15} by using the Kontsevich integral.
%After the work \cite{genus0} we review here, 
%Alekseev and Naef \cite{AN17} gave it by using the Knizhnik--Zamolodchikov connections. 
In \S\ref{subsec:zero} we sketch the proof of the correspondence 
between the formality of the Turaev cobracket and the Kashiwara--Vergne 
problem in genus $0$. The formulation of our theorems in positive genus 
is stated in \S\ref{subsec:HKV}.
\par
In this paper we confine ourselves to a compact connected oriented surface. 
Let $S$ be a compact connected oriented surface with non-empty boundary.
It is classically known that $S$ is classified as an oriented $\smooth$ manifold 
by its genus and the number of its boundary components. 
So we denote by $\Sigma_{g, n+1}$ a compact connected oriented surface
of genus $g$ with $n+1$ boundary components. Such an $S$ is 
diffeomorphic to $\Sigma_{g, n+1}$ for some $g, n \geq 0$. 
The fundamental group $\pi_1(\Sigma_{g,n+1})$ is a free group of rank
$2g+n$. 
\par
We conclude this section by introducing 
a symbol associated with an associative algebra.
Let $A$ be an associative (topological) algebra over a commutative ring $R$
with unit. 
Then we denote by $|A|$ the quotient of $A$ by the commutator $[A, A]$, 
(the closure of) the $R$-submodule generated by the set 
$\{ab-ba; \,\, a, b \in A\}$, and by $|\cdot|: A \to |A|$, $a \mapsto |a|$, 
the quotient map. 
If we regard $A$ as a Lie algebra by the commutator $[a,b] := ab-ba$, 
$a, b \in A$, then $|A|$ equals the abelianization of the Lie algebra $A$. 
A priori, the quotient $|A|$ is just an $R$-module.
For example, if $A = \widehat{T}(V) = \prod^\infty_{m=0}V^{\otimes m}$, 
the completed tensor algebra generated by a vector space $V$ over 
a field of characteristic $0$, then the quotient $|A|$ equals the cyclic coinvariants 
in the completed tensor algebra: $|A| = \prod^\infty_{m=0}
(V^{\otimes m})_{\Z/m}$, where the cyclic group $\Z/m$ acts on the 
space $V^{\otimes m}$ by cyclic permutation of the slots. 
Moreover, if $G$ is a group, we have 
$|RG| = R (G/\text{conj})$, the free $R$-module generated by the set of 
conjugacy classes $G/\text{conj}$ in $G$.
For example, we have $|R\pi_1(S)| = R \hat\pi(S)$. 
\par
\medskip
First of all, the author thanks Athanase Papadopoulos for giving him a chance 
to take part in a tribute to Vladimir Turaev, 
whose works have a great influence
in the mathematics of the twentieth and twenty-first centuries.
Next he thanks Florian Naef for correction and simplification of arguments
in \S\ref{sec:homol}. In particular, the proof of the homotopy invariance of 
the framed Turaev cobracket in \S\ref{subsec:inv} is due to him.
Finally he thanks Yusuke Kuno, who kindly read the whole manuscript carefully, 
and let the author know some typos and helpful suggestions for improvement. 
\par
The present research is supported in part by the grants JSPS KAKENHI 26287006, 18K03283 and 18KK0071.\par

\section{The Turaev cobracket and its framed variants}
\label{sec:def}

\subsection{The original definition of the Turaev cobracket}
\label{subsec:orig}

Let $S$ be a connected oriented surface, i.e., a connected oriented 
$2$-dimensional $\smooth$ manifold possibly with boundary. 
We denote by $\hat\pi$ the free homotopy set of free loops on $S$,
$\hat\pi = \hat\pi(S) := [S^1, S]$.
Since $S$ is connected, it equals the set of conjugacy classes 
$\hat\pi = \pi_1(S)/\text{conj}$ in the fundamental group $\pi_1(S)$.
Moreover, if we denote by $\IntS$ the interior of $S$, then 
we have $\hat\pi(S) = \hat\pi(\IntS)$. 
For any commutative ring $R$ with unit, we have 
$|R\pi_1(S)| = R\hat\pi(S)$, where the latter means 
the free $R$-module generated by the free homotopy set $\hat\pi$.
Throughout this chapter we write simply $\pi = \pi_1(S)$. 
\par

In general, any continuous map from a $1$-dimensional manifold $T$ 
to the surface $S$ is homotopic to a generic immersion, a $\smooth$ 
immersion at worst with transverse double points. 
Two generic immersions are homotopic if and only if they are deformed 
to each other 
by iteration of the following 3 kinds of moves $(\omega 1)$, $(\omega 2)$, $(\omega 3)$ and isotopy: The move $(\omega1)$ is a birth-death of a monogon coming from a cusp,
$(\omega2)$ a birth-death of a bigon coming from a tangential double point,
and $(\omega3)$ a jumping over a double point 
coming from a transverse triple point. 
For details, see \cite{Go86} 5.6 Lemma.
\par
The moves $(\omega2)$ and $(\omega3)$ are regular-homotopies, but
$(\omega 1)$ is not so. One can slide a monogon along the given generic immersed loop 
by using the moves $(\omega2)$, $(\omega3)$ and isotopy. 
Hence two generic immersed loops are regular-homotopic if and only if 
they are deformed to each other 
by iteration of the moves $(\omega 2)$, $(\omega 3)$ and isotopy. \par
Now we recall the definition of the Turaev cobracket for a generic immersed loop
$\alpha: S^1=\R/2\pi\Z \to S$. 
We denote by $D_\alpha$ the set of pairs of parameters mapped to double points:
$D_\alpha := \{(t_1,t_2) \in S^1\times S^1; \,\, \alpha(t_1) = \alpha(t_2),\,
t_1\neq t_2\}$. Then the Turaev cobracket $\delta\alpha \in \Z\hat\pi\otimes\Z\hat\pi = |\Z\pi|\otimes |\Z\pi|$ is defined by 
$$
\delta \alpha := \sum_{(t_1, t_2) \in D_\alpha}
\varepsilon(\overset\cdot\alpha(t_1), \overset\cdot\alpha(t_2))
(\alpha\vert_{[t_1, t_2]})\otimes (\alpha\vert_{[t_2, t_1]})
\in \Z\hat\pi\otimes\Z\hat\pi.
$$
Here $\varepsilon(\overset\cdot\alpha(t_1), \overset\cdot\alpha(t_2))
\in \{\pm1\}$ is the local intersection number of the velocity vectors 
with respect to the orientation of the surface $S$, and $\alpha\vert_{[t_1, t_2]}
\in \hat\pi$ is the free homotopy class of the segment of $\alpha$ 
restricted to the interval on $S^1$ running from $t_1$ to $t_2$ in the positive 
direction. One can prove that $\delta\alpha$ is invariant under 
the moves $(\omega2)$, $(\omega3)$ and isotopy. 
This follows from straightforward arguments, and 
will be proved later in \S3.2 by using twisted homology. 
Hence the operation $\delta$ is regular-homotopy 
invariant. 
The following is a typical example of the Turaev cobracket.
\begin{lem}\label{lem:iterate}
For any simple closed curve $C$ in the interior $\IntS$ and $m \in \Z$, 
we have
$$
\delta(C^m) = 0 \in |\Z\pi|\otimes |\Z\pi|.
$$
\end{lem} 
\begin{proof} Since $C\inv$ is also a simple closed curve 
and $C^0$ is trivial, 
it suffices to show the lemma for $m > 0$. %\par
As a representative for $C^m$, we can take a spiral with 
$m-1$ self-intersection points. Hence $\delta(C^m)$ 
equals
$
\sum^{m-1}_{k=1} C^k\otimes C^{m-k} - C^{m-k}\otimes C^{k}
= 0
$, 
which proves the lemma.
\end{proof}

\par
It is easy to see that $\delta$ is not invariant 
under the move $(\omega 1)$. 
For example, if we insert a monogon in the positive direction 
to $\alpha$, 
then we have to add an extra term $\alpha\otimes \mathbf{1} - \mathbf{1} \otimes \alpha$ 
to the original $\delta\alpha$.  
Here $\mathbf{1} = |1| \in \hat\pi$ is 
the free homotopy class of the constant loop $1$. 
To avoid this ambiguity, Turaev's original definition \cite{Tu91} takes 
the quotient $\Z\hat\pi/\Z\mathbf{1}$. 
Then $\delta\alpha\bmod{\Z\mathbf{1}}$ is homotopy invariant, that is, 
the cobracket 
$$
\delta: \Z\hat\pi/\Z\mathbf{1} \to (\Z\hat\pi/\Z\mathbf{1})^{\otimes 2}
$$
is a well-defined operation. Moreover Turaev \cite{Tu91}
observed that the cobracket together with 
the Goldman bracket $[-,-]$ makes the quotient $\Z\hat\pi/\Z\mathbf{1}$ a Lie bialgebra 
in the sense of Drinfel'd \cite{Drinfeld}. 
Since the constant loop $\mathbf{1}$ is in the center of the Goldman bracket,
the Goldman bracket descends to the quotient $\Z\hat\pi/\Z\mathbf{1}$.  
In particular, the cobracket 
satisfies the compatibility condition with the bracket, i.e., 
\begin{equation}\label{eq:compatibility}
\delta [\alpha, \beta] = \ad(\alpha)(\delta\beta) - 
\ad(\beta)(\delta\alpha) \in (\Z\hat\pi/\Z\mathbf{1})^{\otimes 2}
\end{equation}
for any $\alpha$ and $\beta \in \hat\pi$. 
Here $\ad(\alpha) \in \mathrm{End}((\Z\hat\pi/\Z\mathbf{1})^{\otimes 2})$ 
is the adjoint action defined by $\ad(\alpha)(\beta_1\otimes \beta_2) 
:= [\alpha, \beta_1]\otimes\beta_2 + \beta_1\otimes [\alpha, \beta_2]$
for $\beta_1, \beta_2 \in\hat\pi$. 
This condition can be regarded as the cocycle condition of $\delta$. 
Thus the Lie bialgebra $(\Z\hat\pi/\Z\mathbf{1}, [-,-], \delta)$ is called 
the {\it Goldman--Turaev Lie bialgebra} of the surface $S$. 
It is naturally isomorphic to that of the interior $\IntS$. 
\par
Later Chas \cite{Cha04} proved the Lie bialgebra is involutive.
%\par
This survey covers only some algebraic aspects of the Turaev cobracket.
But we should remark that 
a geometric approach to the cobracket and the Goldman bracket 
by Chas and her coworkers
yields fruitful results on surface topology
\cite{Cha04, Cha10,ChaGa,ChaKr10,ChaKr16, ChaLa}. 
%For example, characterization and statical...
\par

\subsection{Framed Turaev cobrackets}
\label{subsec:frame}

As will be reviewed in \S\ref{subsec:log}, 
we can take the logarithms of some kinds of mapping classes of the surface
$S$ in a completion of $|\K\pi|$. The subset consisting of these 
logarithms essentially equals the image of the Johnson homomorphism.
It is an important problem to describe the Johnson image.
The Turaev cobracket vanishes on these logarithms, 
so that we can derive a constraint on  
the Johnson image from the cobracket. It includes the Morita trace \cite{MoAQ}, 
but not the Enomoto--Satoh trace. In order to recover the Enomoto--Satoh
trace \cite{ES} we need to define a variant  
of the Turaev cobracket with values in $|\K\pi|$ itself \cite{Ka15}.
To do this, we consider a regular-homotopy variant of the Turaev cobracket.
In \cite[\S18]{Tu91} Turaev already studied such a variant. 
But we adopt a slightly different formulation using a framing of the surface $S$. 
\par
From now on, we assume $S$ is compact with non-empty boundary.
In particular, the tangent bundle $TS$ is trivial. We mean by a {\it framing} 
the homotopy class of an orientation-preserving isomorphism 
$TS \overset\cong\to
S\times \R^2$ of vector bundles over $S$. 
This induces a homeomorphism from the complement of the zero section 
$0(S)$ in the total space $TS$ onto the product $S\times (\R^2\setminus\{0\})$.
The set $F(S)$ of all framings  is an affine set modeled on the first cohomology group 
$H^1(S; \Z)$. For any framing $f \in F(S)$ one define the {\it rotation number} 
$\rot_f\alpha \in \Z$ 
of an immersed loop $\alpha: S^1 \to S$ with respect to $f$. 
This is defined by
the mapping degree
$$
\rot_f\alpha := \deg\left(S^1 
\overset{\overset\cdot\alpha}\to TS \setminus 0(S)
\overset{f}\cong S\times (\R^2\setminus\{0\}) \overset{\text{proj}}\to 
\R^2\setminus\{0\}\right) \in \Z.
$$
The move $(\omega1)$ means an insertion of a monogon. 
If we insert a monogon in the positive direction to $\alpha$, 
then $\rot_f\alpha$ increases by $+1$, while 
we have to add an extra term $\alpha\otimes \mathbf{1} - \mathbf{1} \otimes \alpha$ 
to the original $\delta\alpha$.  Hence 
$\delta\alpha + (\rot_f\alpha)(\mathbf{1}\otimes\alpha - \alpha\otimes \mathbf{1})$
is invariant under $(\omega 1)$, so is homotopy invariant, 
by which we define the {\it framed Turaev cobracket} $\delta^f$ 
with respect to a framing $f$
$$
\delta^f: |\Z\pi| \to |\Z\pi|^{\otimes 2}
$$
 \cite{genus0}\cite{highergenus}.
In other words, $\delta^f\alpha$ is defined to be the $\delta$ of 
a representative of $\alpha$ whose rotation number $\rot_f$ 
vanishes. Similarly $(|\Z\pi|, [-,-], \delta^f)$ is an involutive Lie
bialgebra. 
In particular, we have 
\begin{lem}\label{lem:cocycle}
$\delta^f$ is a cocycle with respect to the Goldman bracket, i.e., 
$$
\delta^f([\alpha, \beta]) = \ad(\alpha)(\delta^f\beta) 
- \ad(\beta)(\delta^f\alpha)
$$
for any $\alpha, \beta \in \hat\pi$.
\end{lem}
Thus we obtain the {\it framed Goldman--Turaev Lie bialgebra}
$(|\Z\pi|, [-,-], \delta^f)$ of the framed surface $(S, f)$. 
Since $\delta^f(\mathbf{1}) = 0$, $\delta^f$ can be regarded as a map
$|\Z\pi|/\Z\mathbf{1} \to |\Z\pi|^{\otimes 2}$. 
%The cobracket $\delta^f$ has a similar property 
%to \eqref{eq:Tconti}, so that it descends to 
%the completion $|\widehat{\K\pi}|$ and induces the associated graded
%$\gr(\delta^f)$ on $\gr(|\K\pi|)$. In particular, we can consider 
%the formality problem of $\delta^f$ with respect to the filtration 
%$\{|{I\pi}^p|\}_{p=0}^\infty$. 
As will be reviewed in \S\ref{sec:form}, 
the formality problem of $\delta^f$ 
is closely related to the Kashiwara--Vergne problem.
\par
\medskip
We conclude this subsection by reviewing the topological classification 
of framings. The mapping class group $\mathcal{M}(S)$ of the surface $S$ 
is defined to be the group of path-connected components of the group of 
orientation-preserving diffeomorphisms 
fixing the boundary $\partial S$ {\it pointwise}.
The group $\mathcal{M}(S)$ acts naturally on the set $F(S)$ by 
$\rot_{f\varphi}(\alpha) = \rot_f(\varphi\circ\alpha)$ 
for any $\varphi \in \mathcal{M}(S)$ and $\alpha \in \hat\pi$. 
The rotation number along each boundary loop is an $\mathcal{M}(S)$-invariant of 
framings. On the other hand, 
any framing $f$ induces a spin structure $\mathfrak{s}_f$ on the surface $S$. 
Moreover we can define a new invariant $\tilde A(f)$ of a framing $f$,
which is defined to be the non-negative greatest common divisor of the set 
$\{\rot_f(\alpha); \,\, \text{$\alpha$ is a non-separating simple closed curve in $S$}\}$.
If $g = \mathrm{genus}(S)$ is greater than $1$, we have $\tilde A(f) = 1$ for any framing $f$ \cite[Lemma 2.4]{Ka17}. 

\begin{thm}[\cite{J80a, Ka17}]\label{thm:fr}
Two framings $f$ and $f'$ of the surface $S$ 
are in the same $\mathcal{M}(S)$-orbit 
if and only if \begin{itemize}
\item All the rotation numbers along boundary loops coincide.
\item The induced spin structures $\mathfrak{s}_f$ and 
$\mathfrak{s}_{f'}$ are in the same $\mathcal{M}(S)$-orbit.
\item $\tilde A(f) = \tilde A(f')$.
\end{itemize}
\end{thm}
If $g \geq 2$, the theorem is essentially due to Johnson \cite{J80a},
where he proved that a spin structure is classified by the Arf invariant
modulo the mapping class group $\mathcal{M}(S)$. 
The orbit space $F(S)/\mathcal{M}(S)$ with prescribed boundary data
is an infinite set if and only if $g=1$ and 
the rotation number of each boundary loop equals $-1$.
Only in this case, it is possible that there is no solution of the Kashiwara--Vergne
problem associated with the framed surface (Theorem \ref{thm:exist}).
\par
We can consider the relative version for Theorem
\ref{thm:fr}, i.e., where our homotopy fixes 
a prescribed framing $TS\vert_{\partial S} \cong \partial S\times \R^2$
on the boundary. This version was already established by 
Randal-Williams \cite{RW14} for $g \geq 2$. See \S2.4 in \cite{Ka17} 
for $g=1$.\par

%\subsection{Based analogues}
%\label{subsec:based}

%$\mu$, $\mu^f$, 

\section{Homological descriptions of the cobracket}
\label{sec:homol}
In this section we give a homological description of the Turaev cobracket
after R.\ Hain. Our approach is slightly different from Hain's original one \cite{Hain18}.

\subsection{A relative twisted chain $s_\alpha$}\label{sec:}

We denote
$\Rp := \mathopen]0, +\infty\mathclose[$ and
$\Rnn := \mathopen[0, +\infty\mathclose[ \subset \R$.
Let $E \to B$ be a $\smooth$ $\R$-vector bundle of rank $2$ with 
Riemannian metric $\Vert\cdot\Vert$. We identify the base space $B$ 
with the zero section in the total space $E$.
The multiplicative group $\Rp$ acts on 
the complement $E_0 := E\setminus B = \{e \in E; \,\, \Vert e\Vert \neq 0\}$
by scalar multiplication. 
The orbit space of the action $\Sph(E) := E_0/\Rp \cong \{e \in E; \,\,\Vert e\Vert
= 1\}$ is an $S^1$ bundle over $B$. 
For any $x \in E_0$, we denote $[x] := x\bmod \Rp \in \Sph(E)$ and 
$[x]_0 := \Rnn\cdot x \subset E$. 
Then the space 
$$
Q(E) := \{([x], e)\in \Sph(E)\times E; \,\, e \in [x]_0\}
$$
is naturally diffeomorphic to the product $\Sph(E) \times \Rnn$
by the map $([x], \lambda) \in \Sph(E)\times \Rnn \mapsto ([x], \lambda \Vert x\Vert\inv x) \in Q(E)$, and admits a natural projection $\varpi: Q(E) \to E$, 
$([x], e) \mapsto e$, which is the {\it real oriented blow-up of $E$ along $B$}.
If $\partial B = \emptyset$, we have $\partial Q(E) = \Sph(E) \times \{0\}$.
\par
In this section we write simply $M = \IntS$, the interior of the surface $S$. 
%Let $M$ be an open connected oriented $\smooth$ $2$-manifold.
We denote by $\Delta_M \subset M\times M$ the diagonal set, whose 
normal bundle is given by 
$$
E := N_{M\times M/\Delta_M} = (TM\times_MTM)/\Delta(TM) \cong TM,
\quad (u, v) \mapsto v-u.
$$
Choose a tubular neighborhood of the diagonal $h : E \hookrightarrow M\times M$.
Then we glue the space $Q(E)$ with the configuration space 
$C_2(M) := M\times M \setminus \Delta_M$ 
by the map $h$:
$$
\hat C_2(M) := Q(E) \cup_h C_2(M),
$$
which admits a natural projection $\varpi: \hat C_2(M) \to M\times M$ 
induced by the map $Q(E) \overset\varpi\to E \overset{h}\to h(E) \subset M\times M$.
\par
We identify $\varpi\inv(\Delta_M)$ with $\Sph(TM)$ as above. 
Let $f$ be a framing of $S$.
We denote by $\mathbf{f}_f := f^*1_{S^1} \in H^1(\Sph(TM); \Z)$ 
the pullback of the positive generator $1_{S^1} \in H^1(S^1; \Z)$ 
by the framing $f$. Hain uses the symbol $\xi$ for a framing, 
and writes $f_\xi$ for the cohomology class $\mathbf{f}_f$ in \cite[\S3]{Hain18}. 
%The connecting homomorphism $\delta^*:
%H^1(\Sph(TM)) \to H^2(\hat C_2(M), \Sph(TM))$ maps the class $\mathbf{f}_f$ to 
%$$
%\delta^*\mathbf{f}_f \in H^2(\hat C_2(M), \Sph(TM); \Z).
%$$
On the other hand, we introduce a local system $\Ss$ on $M\times M$ 
whose stalk is defined by 
$$
\Ss_{(p_1, p_2)} := \Z\Pi M(p_1, p_2) \otimes \Z\Pi M(p_2, p_1)
$$
for any $(p_1, p_2) \in M\times M$. Here we denote by 
$\Pi M(p_1, p_2)$ the homotopy set of paths $[([0,1], 0, 1), (M, p_1, p_2)]$.
As in the previous section, we denote $\hat\pi = [S^1, M] = [S^1, S]$.
Then one can define a natural map $\nu: H_0(\Sph(TM); \varpi^*\Ss) \to 
\Z\hat\pi\otimes \Z\hat\pi$, and the composite 
%$(\mathbf{f}_f\cap)\circ\pa_*:$
\begin{equation}
%&&(\mathbf{f}_f\cap)\circ\pa_* : \nonumber\\
%&&
H_2(\hat C_2(M), \Sph(TM); \varpi^*\Ss ) 
\overset{\pa_*}\to 
H_1(\Sph(TM); \varpi^*\Ss)
\overset{\mathbf{f}_f\cap}\to 
H_0(\Sph(TM); \varpi^*\Ss) \overset\nu\to 
\Z\hat\pi\otimes \Z\hat\pi.
\label{eq:cap}
\end{equation}
\par
Now we consider immersed curves in a $2$-dimensional $\R$-vector space
$V (\cong \R^2)$ with a Euclidean metric $\Vert\cdot\Vert$.
The quotient space of the scalar multiplication $\Sph(V) = (V\setminus\{0\})/\Rp$ 
is diffeomorphic to the circle $S^1 = \R/2\pi\Z$. We choose the map 
$$
\Delta_V\times V= V\times V \to V\times V, \quad 
(x,y) \mapsto (x-y, x+y), \,\, (2\inv(u+v), 2\inv(v-u)) \leftarrow\hskip-2mm\vert
(u,v)
$$
as a tubular neighborhood of the diagonal $\Delta_V$ in $V\times V$. 
The space $\hat C_2(V)$ is given by 
$$
\hat C_2(V) = \{([x], u, v) \in \Sph(V)\times V\times V; \,\, v-u \in [x]_0\}, 
$$
and the map $\varpi: \hat C_2(V) \to V\times V$ by $([x], u, v) \mapsto (u,v)$. 
\par
We study immersed curves in the $2$-manifold $V$ 
in the following 2 cases. Let $I \subset \R$ be a non-empty interval. 
First we consider a single $\smooth$ immersion $\alpha: I \to V$, 
and triangles $T_+:= \{(t_1,t_2) \in I\times I; \,\, t_1\leq t_2\}$ and 
$T_-:= \{(t_1,t_2) \in I\times I; \,\, t_1\geq t_2\}$. We denote 
$\Delta_\pm := T_\pm \cap \Delta_I$, where $\Delta_I \subset I\times I$
is the diagonal. 
\begin{lem}\label{lem:alphaalpha}
The map $(\alpha, \alpha):  I\times I \to V\times V$, $(t_1, t_2) \mapsto 
(\alpha(t_1), \alpha(t_2))$, restricted to the triangle $T_\pm$, canonically lifts to $\hat C_2(V)$ near the diagonal
$\Delta_\pm$. 
\end{lem}
\begin{proof} The lift $\hat\alpha_\pm: T_\pm \to \hat C_2(V)$ is explicitly given by
$$
\hat \alpha_\pm(t_1, t_2) := \begin{cases}
([\alpha(t_2)-\alpha(t_1)], \alpha(t_1), \alpha(t_2)), & \text{if $t_1\neq t_2$,}\\
([\pm\overset\cdot\alpha(t_0)], \alpha(t_0), \alpha(t_0)), & 
\text{if $(t_1, t_2) =: (t_0, t_0) \in \Delta_\pm$.}
\end{cases}
$$
The first line makes sense near the diagonal $\Delta_I$ since $\alpha$ is an immersion. On the other hand, 
$$
[\alpha(t_2)-\alpha(t_1)] = [
\dfrac{t_2-t_1}{|t_2-t_1|}
\int^1_0\overset\cdot\alpha((1-s)t_1+st_2) ds]
\in \Sph(V) \cong S^1
$$
is $\smooth$ in $(t_1, t_2)$ near $(t_0,t_0)$ since $\overset\cdot\alpha(t_0)
\neq 0$. This proves that the map $\hat\alpha_\pm$ is a well-defined lift of the map
$(\alpha, \alpha)$ near the diagonal $\Delta_\pm$. 
\end{proof}
\par
Next we consider two curves intersecting transversely at a point.
Let $I \subset \R$ be an interval including the origin $0 \in \R$, 
and $\alpha_1$ and $\alpha_2: I \to V$ $\smooth$ immersions with 
$\alpha_1(0) = \alpha_2(0) (=: p_0 \in V)$. Moreover we assume that
the vectors $\overset\cdot\alpha_1(0)$ and $\overset\cdot\alpha_2(0)$ 
are linearly independent in $T_{p_0}V$ with local intersection number 
$\varepsilon_{p_0} = \varepsilon(\overset\cdot\alpha_1(0), 
\overset\cdot\alpha_2(0)) \in \{\pm1\}$. 
Then we consider the map $\alpha := (\alpha_1, \alpha_2):
I \times I \to V\times V$, $(t_1, t_2) \mapsto (\alpha_1(t_1), \alpha_2(t_2))$.
%
\begin{comment}
\begin{lem}\label{lem:alpha12}
The local intersection number of the map $\alpha$ 
and the diagonal $\Delta_V$ at the point $(0,0) 
\in I\times I$ equals $- \varepsilon_{p_0}$. 
\end{lem}
\begin{proof}
Since $\varepsilon_{p_0}\cdot ((\overset\cdot\alpha_1(0), \overset\cdot\alpha_1(0)), (\overset\cdot\alpha_2(0), \overset\cdot\alpha_2(0)))$ is 
positive in $T_{(p_0, p_0)}\Delta_V$, $\epsilon_{p_0}$ times 
the local intersection number of 
the map $\alpha$ and the diagonal equals the sign in $T_{(p_0, p_0)}V\times V$ 
of the basis 
$((\overset\cdot\alpha_1(0), 0), (0, \overset\cdot\alpha_2(0)), 
(\overset\cdot\alpha_1(0), \overset\cdot\alpha_1(0)), 
(\overset\cdot\alpha_2(0), \overset\cdot\alpha_2(0)))$ or equivalently 
that of the basis
$((\overset\cdot\alpha_1(0), 0), (0, \overset\cdot\alpha_2(0)), 
(0, \overset\cdot\alpha_1(0)), (\overset\cdot\alpha_2(0), 0))$, 
which equals $- {\varepsilon_{p_0}}^2 = -1$. This proves the lemma.
\end{proof}
\end{comment}
%
The map $\alpha$ does not lift to $\hat C_2(V)$. We need to take 
a (real oriented) blow-up of $I\times I$ at the point $(0,0)$
$$
\widehat{I\times I} := \{(\theta, t_1, t_2) \in S^1\times I\times I; \,\,
(t_1, t_2) = (r\cos\theta, r\sin\theta) \,\text{for some $r \geq 0$}\}.
$$
Then a lift $\hat\alpha: \widehat{I\times I} \to \hat C_2(V)$ is defined by 
\begin{equation}\label{eq:alphahat}
\hat\alpha(\theta, t_1, t_2):= 
\begin{cases}
([\alpha_2(t_2)-\alpha_1(t_1)], \alpha_1(t_1), \alpha_2(t_2)), & \text{if $(t_1, t_2) 
\neq (0,0)$,}\\
([-(\cos\theta)\overset\cdot\alpha_1(0)
+ (\sin\theta)\overset\cdot\alpha_2(0)], 
p_0, p_0), & \text{if $(t_1, t_2) = (0,0)$.}
\end{cases}
\end{equation}
The second line makes sense since 
the vectors $\overset\cdot\alpha_1(0)$ and $\overset\cdot\alpha_2(0)$ 
are linearly independent. 
In order to check the smoothness near $S^1\times \{(0,0)\}$, 
we use a local coordinate $\mathopen[0, \delta\mathclose[ \times S^1
\to \widehat{I\times I}$, $(\rho, \theta) \mapsto (\theta, \rho\cos\theta, 
\rho\sin\theta)$, for a sufficiently small $\delta > 0$, 
and recall the formula $\alpha_i(t) - p_0 = t\int^1_0
\overset\cdot\alpha_i(st)ds$, $i = 1,2$.
Then $[\alpha_2(\rho\sin\theta) - \alpha_1(\rho\cos\theta)]$ equals
$$
[(\sin\theta)\int^1_0\overset\cdot\alpha_2(s\rho\sin\theta)ds
- (\cos\theta)\int^1_0
\overset\cdot\alpha_1(s\rho\cos\theta)ds] \in \Sph(V)
$$
which is $\smooth$ in $(\rho, \theta)$ near $S^1\times \{(0,0)\}$. 
This proves the smoothness of the map $\hat\alpha$. %\par
It should be remarked that the map $\hat\alpha$ is 
uniquely determined by $\alpha$. In fact, it must coincide with 
$(\alpha_1, \alpha_2)$ on the complement $\widehat{I\times I}\setminus
(S^1\times\{(0,0)\})$ which is dense in $\widehat{I\times I}$.
\par
Now we go back to the situation where we defined the framed Turaev cobracket.
Let $M$  and $f$ be as above, and 
$\alpha: S^1 \to M$ a generic 
$\smooth$ immersion. The set parametrizing double points of $\alpha$
$$
D_\alpha = \{(t_1, t_2) \in S^1\times S^1\setminus \Delta_{S^1}; \,\,
\alpha(t_1) = \alpha(t_2)\}
$$
is a finite set. Moreover, using the identification $S^1 = \R/2\pi\Z$, we introduce the map 
$q: S^1\times [0,2\pi] \to S^1\times S^1$, $(\theta, t) \mapsto (\theta, \theta+t)$, and we take the real oriented blow-up of $S^1\times [0,2\pi]$ 
at all points in the set $q\inv(D_\alpha)$, which we denote by
$Q_\alpha(S^1\times [0,2\pi])$. Then we have a canonical lift 
$\hat\alpha: Q_\alpha(S^1\times [0,2\pi]) \to \hat C_2(M)$ of 
the map $(\alpha, \alpha)\circ q: S^1\times [0, 2\pi] \to M\times M$, 
$(t_1, t_2) \mapsto (\alpha(t_1), \alpha(t_2))$. 
Here the lift near $S^1\times\{0\}$ (resp. $S^1\times\{2\pi\}$) comes from $\hat\alpha_+$ (resp. $\hat\alpha_-$) in Lemma \ref{lem:alphaalpha}. 
\par
A canonical horizontal lift $\bar s_\alpha: S^1\times [0, 2\pi] \to \Ss$ 
of the map $(\alpha, \alpha) \circ q: S^1\times [0, 2\pi] \to M\times M$ 
is defined by 
$$
\bar s_\alpha(\theta, t) := 
\begin{cases}
\alpha\vert_{[\theta, \theta+t]}\otimes \alpha\vert_{[\theta+t, \theta]},
& \text{if $t \neq 0, 2\pi$},\\
1\otimes \alpha, & \text{if $t = 0$},\\
\alpha\otimes 1, & \text{if $t = 2\pi$}.
\end{cases}
$$
Here the subset $[\theta, \theta+t] \subset S^1$ 
is defined to be the closed interval running from $\theta$ to $\theta+t$ in the positive
direction, and $\alpha\vert_{[\theta, \theta+t]}$ is the based homotopy class 
of the segment of $\alpha$ restricted to the interval $[\theta, \theta+t]$. 
We denote by $s_\alpha$ the lift of $\bar s_\alpha$ along the map $\hat \alpha$, 
which is a twisted cycle and defines a twisted homology class 
$$
[s_\alpha] \in H_2(\hat C_2(M), \Sph(TM); \varpi^*\Ss).
$$
\begin{prop}\label{prop:homol}
$$
\nu(\mathbf{f}_f\cap\pa_*[s_\alpha]) = 
\delta\alpha + (\rot_f\alpha)(\mathbf{1}\otimes\alpha - \alpha\otimes \mathbf{1}) \in \Z\hat\pi\otimes \Z\hat\pi.
$$
The right-hand side exactly equals the framed Turaev cobracket $\delta^f\alpha$
with respect to $f$. 
\end{prop}
\begin{proof}
The contribution of the boundary $\partial(S^1\times[0,2\pi])$ equals 
$(\rot_f\alpha)(\mathbf{1}\otimes\alpha - \alpha\otimes \mathbf{1})$. 
In order to compute the contribution of each $(t_1, t_2) \in D_\alpha$, 
we consider the construction \eqref{eq:alphahat} for 
$\alpha_1(t) := \alpha(t_1+t)$ and $\alpha_2(t) := \alpha(t_2+t)$.
Then the mapping degree of the map
$$
\theta \in S^1 \mapsto [-(\cos\theta)\overset\cdot\alpha_1(0)
+ (\sin\theta)\overset\cdot\alpha_2(0)] \in S^1
$$
equals $-\varepsilon(\overset\cdot\alpha_1(0), 
\overset\cdot\alpha_2(0)) = 
-\varepsilon(\overset\cdot\alpha(t_1), 
\overset\cdot\alpha(t_2))$. The orientation of $S^1$ induced from $S^1\times S^1$
is the negative one. Hence the contribution equals 
$\varepsilon(\overset\cdot\alpha(t_1), 
\overset\cdot\alpha(t_2))(\alpha\vert_{[t_1, t_2]})\otimes (\alpha\vert_{[t_2, t_1]})$.
The sum of all these contributions is nothing but the Turaev cobracket.
This proves the proposition.
\end{proof}
\begin{rmk}{\rm 
The difference of this computation from Hain's original one 
in \cite[\S5]{Hain18} comes from the blow-ups on $q\inv(D_\alpha)$. 
}\end{rmk}

\par

\subsection{Homotopy invariance}\label{subsec:inv}

In this subsection we prove the homotopy invariance of the framed 
Turaev cobracket. The present proof is due to Florian Naef.
The key to the proof is the following lemma.
\begin{lem}\label{lem:isom}
The map $\varpi: (\hat C_2(M), \Sph(TM)) \to (M\times M, \Delta_M)$ 
induces an isomorphism of twisted homology groups
$$
\varpi_*: H_*(\hat C_2(M), \Sph(TM); \varpi^*\Ss) \overset\cong\to 
H_*(M\times M, \Delta_M; \Ss).
$$
\end{lem}
\begin{proof} Let $N \subset M\times M$ be a closed tubular 
neighborhood of the diagonal $\Delta_M$, and $N^\circ$ its interior.
Then the pair $(M\times M \setminus N^\circ, \pa N)$ is 
homotopy equivalent to $(\hat C_2(M), \Sph(TM))$ 
by a deformation supported near $N$. By excision, we have
$H_*(M\times M \setminus N^\circ, \pa N; \Ss) \cong
H_*(M\times M, N; \Ss) \cong H_*(M\times M, \Delta_M; \Ss)$.
Hence we obtain an isomorphism $H_*(\hat C_2(M), \Sph(TM); \varpi^*\Ss)
\cong H_*(M\times M, \Delta_M; \Ss)$, which can be realized by the induced
map $\varpi_*$. 
\end{proof}

By this lemma we can consider the map
\begin{equation}\label{eq:htpy}
\nu\circ (\mathbf{f}_f\cap)\circ\pa_*\circ{\varpi_*}\inv: 
H_2(M\times M, \Delta_M; \Ss) \to \Z\hat\pi\otimes \Z\hat\pi.
\end{equation}
The isomorphism $\varpi_*$ maps the relative homology class 
$[s_\alpha] \in H_2(\hat C_2(M), \Sph(TM); \varpi^*\Ss)$ 
to $[\bar s_\alpha] \in H_2(M\times M, \Delta_M; \Ss)$.
Since the homology class $[\bar s_\alpha]$ is homotopy invariant,
$$
\nu\circ (\mathbf{f}_f\cap)\circ\pa_*\circ{\varpi_*}\inv([\bar s_\alpha]) 
= \delta\alpha + (\rot_f\alpha)(\mathbf{1}\otimes\alpha - \alpha\otimes \mathbf{1}) \in \Z\hat\pi\otimes \Z\hat\pi
$$
is also homotopy invariant, as was to be shown.\par
Here we remark that this proof establishes the invariance of the framed Turaev 
cobracket also under the move $(\omega 1)$. 

\section{Formality of the Goldman bracket}
\label{sec:FG}

In this section we review the formality of the Goldman bracket 
for $\Sigma_{g, n+1}$ (with $g, n \geq 0$) over $\K$, 
a field of characteristic $0$.
%Some topics in this section overlap with those in our joint 
%previous survey \cite{KK16} with Y.\ Kuno. 
But the filtration we consider here is different from that in \cite{KK16}. 
So we need to re-formulate the problem from the beginning.\par

\subsection{Group-like expansions.}
\label{subsec:exp}

For any free group $\pi$ of finite rank, the formality problem of the 
group ring $\K\pi$ is classically solved by group-like expansions. 
See, for example, \cite{Bou71} and also \cite{KK16}.
The group ring $\K\pi$ admits the augmentation map
$\varepsilon: \K\pi \to \K$
which maps a formal finite sum 
$\sum_{\gamma\in \pi} a_\gamma\gamma$, 
$a_\gamma \in \K$, to the finite sum 
$\sum_{\gamma\in \pi} a_\gamma$. 
The augmentation ideal $I\pi := \ker\varepsilon$ defines 
a natural decreasing filtration $\{(I\pi)^m\}_{m=0}^\infty$ on $\K\pi$, 
which induces the (completed) associated graded quotient
$
\gr(\K\pi) := \prod^\infty_{m=0}(I\pi)^m/(I\pi)^{m+1}
$
and the completed group ring 
$
\widehat{\K\pi} := \varprojlim_{m\to\infty}\K\pi/((I\pi)^m).
$
The group ring $\K\pi$ is a Hopf algebra with coproduct $\Delta$
satisfying $\Delta\gamma = \gamma\otimes\gamma$ for any 
$\gamma \in \pi$. The coproduct $\Delta$ descends to 
both of the algebras $\gr(\K\pi)$ and $\widehat{\K\pi}$,
and make them complete Hopf algebras.
The formality problem for $\K\pi$ 
asks whether the complete Hopf algebras 
$\gr(\K\pi)$ and $\widehat{\K\pi}$ are isomorphic to each other.
\par
In any complete Hopf algebra $A$, the Lie-like elements
$\mathrm{Lie}(A) := \{a \in A; \,\,
\Delta(a) = a\otimes 1 + 1\otimes a\}$ and the group-like
elements $\mathrm{Gr}(A):= \{g \in A\setminus\{0\}; \,\,
\Delta(g) = g\otimes g\}$ are in one-to-one correspondence 
by the exponential $\exp: \mathrm{Lie}(A) \to \mathrm{Gr}(A)$, 
$a \mapsto \exp(a) := \sum^\infty_{m=0}\frac1{m!}a^m$, 
and the logarithm $\log: \mathrm{Gr}(A)\to \mathrm{Lie}(A)$,
$g \mapsto \log(g) := \sum^\infty_{m=1}\frac{(-1)^{m+1}}{m}
(g-1)^m$.  The subset $\mathrm{Gr}(A)$ is a subgroup 
of the multiplicative group of the algebra $A$, and $\mathrm{Lie}(A)$
is a Lie subalgebra of $A$. \par
Let $H$ be the first homology group of the group $\pi$, 
$H := H_1(\pi; \K) = (\pi/[\pi,\pi])\otimes\K$, and 
we denote by $[\gamma] \in H$ the homology class 
of $\gamma \in \pi$. The completed tensor algebra 
$\widehat{T}(H) := \prod^\infty_{m=0}H^{\otimes m}$ 
has a natural filtration $\widehat{T}(H)_{\geq p}
:= \prod^\infty_{m=p}H^{\otimes m}$, $p \geq 0$, 
and a natural coproduct $\Delta$ satisfying 
$\Delta[\gamma] = [\gamma]\otimes 1 + 1\otimes [\gamma]$
for any $\gamma \in \pi$, which makes $\widehat{T}(H)$ 
a complete Hopf algebra. 
The Lie-like element $\mathrm{Lie}(\widehat{T}(H))$ equals the 
completed free Lie algebra over the vector space $H$.
We write simply $\mathrm{Lie}(\widehat{T}(H))_{\geq p}
:= \mathrm{Lie}(\widehat{T}(H))\cap \widehat{T}(H)_{\geq p}$
for any $p \geq 1$. \par
A group-like expansion $\theta: \pi \to \widehat{T}(H)$ is 
a group homomorphism $\theta: \pi \to \mathrm{Gr}(\widehat{T}(H))$
satisfying $\log \theta(\gamma) \equiv [\gamma] \pmod{
\mathrm{Lie}(\widehat{T}(H))_{\geq 2}}$
for any $\gamma \in \pi$. Its linear extension induces isomorphisms
of complete Hopf algebras $\theta: \gr(\K\pi) \to \widehat{T}(H)$ and 
$\theta: \widehat{\K\pi} \to \widehat{T}(H)$. The former one 
is independent of the choice of group-like expansions, so that 
we identify $\gr(\K\pi) = \widehat{T}(H)$. Thus any group-like expansion 
$\theta$ induces an isomorphism of complete Hopf algebras 
$\theta: \widehat{\K\pi} \to \widehat{T}(H)$ whose associated graded equals 
the identity on $\gr(\K\pi)$. In other words, the formality problem 
for the Hopf algebra $\K\pi$ is solved. 
%\par
This formulation is enough for two extreme cases $\Sigma_{0,n+1}$ and $\Sigma_{g,1}$, while we replace the filtration $\{{I\pi}^m\}_{m=0}^\infty$
by a new filtration $\{\K\pi(m)\}_{m=0}^\infty$ defined below 
in order to obtain a formal description of the Goldman--Turaev Lie bialgebra.
\par
We number the boundary components from $0$ to $n$: 
$\partial S = \coprod^n_{j=0}\partial_jS$, and choose 
a basepoint $*$ on the $0$-th boundary component $\partial_0S$. 
Denote by $\pi$ the fundamental group $\pi_1(S, *)$, which 
is a free group of rank $2g+n$. It admits 
free generators $\alpha_i$, $\beta_i$, $\gamma_j$, 
$1 \leq i \leq g$, $1 \leq j \leq n$, such that $\gamma_j$ is freely homotopic 
to the $j$-th boundary $\partial_jS$ with the positive direction and the product
$$
\gamma_0 := \prod^g_{i=1}\alpha_i\beta_i{\alpha_i}\inv{\beta_i}\inv
\prod^n_{j=1}\gamma_j
$$
is the based loop around the $0$-th boundary in the {\it negative} direction. 
Choose a point $*_j$ for each boundary component $\partial_jS$, 
$1\leq j \leq n$.
Then we can take $\gamma_j \in \pi$ as a concatenation of 
a path $\ell_j$ from $*$ to $*_j$, the boundary loop based at
$*_j$ and the inverse ${\ell_j}\inv$. 
Moreover we write $*_0 := * \in \partial_0 S$. 
\par
The filtration $\{\K\pi(m)\}_{m=1}^\infty$,
which we call the {\it weight filtration},  
is defined to be a unique minimal multiplicative filtration of two-sided ideals 
of the algebra $\K\pi$ satisfying the conditions $\K\pi(0) = \K\pi$, 
$\alpha_i-1, \beta_i-1 \in \K\pi(1)$, $1 \leq i \leq g$, and $\gamma_j -1 
\in \K\pi(2)$, $1 \leq j \leq n$. See \cite[\S\S3.1-2]{highergenus}. 
Here we remark this filtration coincides with (twice) the 
filtration $\{(I\pi)^m\}_{m=0}^\infty$ if $n=0$ or $g=0$. 
Even for any $\Sigma_{(g, n+1)}$, the topology on $\K\pi$ induced by the weight filtration 
coincides with that by the filtration $\{(I\pi)^m\}$. 
Hence we have a natural identification 
$
\widehat{\K\pi} = \varprojlim_{m\to\infty} \K\pi/\K\pi(m).
$
The (completed) associated graded quotient
$$
\grwt(\K\pi) := \prod^\infty_{m=0}\K\pi(m)/\K\pi(m+1)
$$
also has a natural complete Hopf algebra structure, and 
is described as follows.\par
We denote 
$$
x_i := [\alpha_i],\, y_i:= [\beta_i] \in H,\,\, 1\leq i\leq g, \quad
z_j:=[\gamma_j] \in H, \,\,1 \leq j \leq n, 
$$
and by $H^{(2)}$ the linear span of the $z_j$'s, which is the annihilator 
of the intersection number $H\times H \to \K$. 
On the other hand, we cap closed disks on the boundary components 
from $1$ to $n$ to obtain a surface $\bar S$ diffeomorphic to $\Sigma_{g,1}$.
The inclusion homomorphism $i_*: H = H_1(S) \to H_1(\bar S)$ is 
a surjection, and its kernel equals $H^{(2)}$. 
In particular, one can identify $H/H^{(2)} = H_1(\bar S)$. 
Denote by $\bar x_i$ and $\bar y_i \in H_1(\bar S) = H/H^{(2)}$ the homology classes of $\alpha_i$ and $\beta_i$, respectively.
Thus we obtain a two-step filtration on the homology group $H = H^{(1)} \supset
H^{(2)} \supset 0$. One can consider the graded quotient $\gr H = (H/H^{(2)})
\oplus H^{(2)}$ with weight $\wt(H/H^{(2)}) =1$ and $\wt(H^{(2)}) = 2$.
Any isomorphism $H \to \gr H$ whose associated graded equals 
the identity on $\gr H$ corresponds to a section of the surjection 
$i_*: H \to H_1(\bar S)$, and vice versa. In particular, 
we often fix a unique isomorphism $H \cong \gr H$
mapping $x_i $ and $y_i$ to  $\bar x_i$ and $\bar y_i$, 
$1 \leq i \leq g$, respectively, and $z_j$ to itself, $1 \leq j \leq n$. 
As was proved in \cite[Proposition 3.12]{highergenus}, we have a natural 
isomorphism of graded Hopf algebras
\begin{equation}\label{eq:grwtTH}
\grwt(\K\pi) = \widehat{T}(\gr H).
\end{equation}
\par
Here we change the definition of a group-like expansion: 
A group-like expansion is an isomorphism $\theta: \widehat{\K\pi}
\to \grwt(\K\pi)$ of complete filtered Hopf algebras whose associated 
graded equals the identity on $\grwt(\K\pi)$. 
A unique group-like expansion $\theta_{\exp}$ 
which maps $\alpha_i$ to $\exp(\bar x_i)$, 
$\beta_i$ to $\exp(\bar y_i)$, $1\leq i \leq g$, and $\gamma_j$ to 
$\exp(z_j)$, $1 \leq j \leq n$, is a typical example, which we call 
the {\it exponential} expansion associated with $\{\alpha_i, \beta_i, \gamma_j\}$. 
\par
A group-like expansion $\theta$ is {\it tangential} if we have $\theta(\gamma_j) 
= g_j \exp(z_j){g_j}\inv$ for some $g_j\in \mathrm{Gr}(\grwt(\K\pi))$,  
$1 \leq j \leq n$. The exponential expansion $\theta_{\exp}$ is tangential. 
The group-like element $g_j$ is uniquely determined 
up to right multiplication by $\exp(\lambda_jz_j)$, 
$\lambda_j \in \K$. If we fix the elements $g_j$'s, then 
one can extend the expansion to paths connecting 
two different boundary components of the surface
uniquely as in \cite[Definition 7.2]{KK16}. 
\par
The most important notion in this section is that of a special expansion.
A tangential expansion is {\it special} if $\log \theta(\gamma_0) 
= \omega$, where we define
$$
\omega : = \sum_i[\bar x_i, \bar y_i] + \sum_jz_j \in \grwt(\K\pi).
$$
Special expansions do exist for any $g, n \geq 0$, 
while the exponential one is not special.  
If $n=0$, a special expansion is called a {\it symplectic} expansion,
which was introduced by Massuyeau \cite{Mas12}. 
The discovery of the notion of a symplectic expansion by Massuyeau 
has been one of the important motives 
of the author's studies in the last decade. 
A special expansion often appears in genus $0$
\cite{AET10, HM00, Mas15}. We should call it a {\it special/symplectic}
expansion, but here we call just a special expansion for simplicity. \par
\begin{rmk}{\rm 
This definition is equivalent to that in \cite[Definition 7.2]{KK16}.
Let $\theta: \widehat{\K\pi} \overset\cong\to \grwt(\K\pi)$ be 
a special expansion. Then the composite of $\log\theta\vert_\pi: 
\pi \to \mathrm{Lie}(\grwt(\K\pi))$ and the natural projection
$\mathrm{Lie}(\grwt(\K\pi)) \to \gr H \to H^{(2)}$ defines a 
splitting $\pi \to H \to H^{(2)}$, from which we obtain a section $s$ of 
the inclusion homomorphism $i_*: H \to H_1(\bar S)$. 
The composite of $\theta$ and the isomorphism $\grwt(\K\pi) \cong \widehat{T}(H)$ induced by the section $s$ satisfies the condition ($\sharp_s$) 
in \cite[(7.3)]{KK16}. Conversely, if we are given a section $s$ of $i_*$ and 
a Magnus expansion with condition ($\sharp_s$), then the composite of 
the Magnus expansion and the isomorphism $\grwt(\K\pi) \cong \widehat{T}(H)$ induced by the section $s$ gives a special expansion in the sense of the present survey.}
\end{rmk}
\par

\subsection{The Goldman bracket}
\label{subsec:gold}

Now we recall the definition of the Goldman bracket. 
Choose generic immersions representing $\alpha$ and $\beta \in \hat\pi$.
In particular, the intersection points $\alpha\cap\beta$ are finite and transverse.
Then the Goldman bracket is defined to be
$$
[\alpha, \beta] := \sum_{p\in \alpha\cap \beta}
\varepsilon_p(\alpha, \beta) |\alpha_p\beta_p| \in \Z\hat\pi
$$
\cite{Go86}. Here $\varepsilon_p(\alpha, \beta)\in \{\pm1\}$ is the local intersection 
number of $\alpha$ and $\beta$ at the intersection point $p$, 
and $\alpha_p$ and $\beta_p \in \pi_1(S, p)$ are 
the based loops with basepoint $p$ along $\alpha$ and $\beta$, respectively.
Moreover $\alpha_p\beta_p$ means the product of the based loops, whose free homotopy 
class is denoted by $|\alpha_p\beta_p| \in \hat\pi$. 
Then the bracket is well-defined, and makes the module $\Z\hat\pi$ 
into a Lie algebra \cite{Go86}. One can check
\begin{eqnarray}
& & [|\K\pi(l)|, |\K\pi(m)|] \subset |\K\pi(l+m-2)|, \quad\text{and} \label{eq:Gfil}\\
& & \delta^f(|\K\pi(m)|) \subset \sum_{a+b = m-2} |\K\pi(a)|\otimes |\K\pi(b)|
\label{eq:Tfil}
\end{eqnarray}
\cite[Proposition 3.13]{highergenus}. Hence the Goldman bracket and 
the Turaev cobracket descend to $|\widehat{\K\pi}|$ and 
define their associated graded 
$$
\aligned
& [-,-]_{\grwt}: \grwt(|\K\pi|)\otimes \grwt(|\K\pi|) \to \grwt(|\K\pi|)\\
& \grwt(\delta^f): \grwt(|\K\pi|) \to \grwt(|\K\pi|)\otimes \grwt(|\K\pi|)\\
\endaligned
$$ 
of degree $-2$. 
It should be remarked that $\grwt(|\K\pi|) = |\grwt(\K\pi)|$. \par
To describe these operations, we prepare some notation.
The intersection number on $H = H_1(S; \K)$ induces a non-degenerate 
skew-symmetric pairing $\langle-, -\rangle: (H/H^{(2)})\times (H/H^{(2)})
\to \K$, which we extend to $\gr H$ by zero on $H^{(2)}$. 
A symmetric operation 
$\mathfrak{z}: H^{(2)}\times H^{(2)}\to H^{(2)}$
is defined by $\mathfrak{z}(z_j, z_k) = \delta_{jk}z_j$, $1 \leq j, k \leq n$.
The sum $z_0 := \sum^n_{j=1}z_j$ is the unit of the operation $\mathfrak{z}$, 
which we extend to $\gr H$ by zero on $H/H^{(2)}$. 
Then the associated graded of the Goldman bracket is given by 
\begin{eqnarray}
[u, v]_{\grwt} &=& 
\sum_{i,j}\langle u_i, v_j\rangle|u_{i+1}\cdots u_lu_1\cdots u_{i-1}
v_{j+1}\cdots v_mv_1\cdots v_{j-1}|\label{eq:Gform}\\
&& + \sum_{i,j}|\mathfrak{z}(v_j, u_i)u_{i+1}\cdots u_lu_1\cdots u_{i-1}
v_{j+1}\cdots v_mv_1\cdots v_{j-1}|\nonumber\\
&& - \sum_{i,j}|\mathfrak{z}(u_i, v_j)v_{j+1}\cdots v_mv_1\cdots v_{j-1}u_{i+1}\cdots u_lu_1\cdots u_{i-1}|.\nonumber
\end{eqnarray}
Here $u = u_1\cdots u_l$, $v = v_1\cdots v_m \in \widehat{T}(\gr H)$ 
for $u_i, v_j \in \gr H$. 
For the framed Turaev cobracket, we need one more symbol: 
For a framing $f \in F(S)$ we define $c_f: \gr H \to \K$ by 
$c_f(\bar x_i) = c_f(\bar y_i) = 0$, $1 \leq i \leq g$, and 
$c_f(z_j) = \rot_f(\partial_jS) + 1$, $1 \leq j \leq n$. 
Here we remark $z_0 = \sum^n_{j=1}z_j$ satisfies 
$c_f(z_0) = 2-2g-n-1 - \rot_f(\partial_0S) + n = 1-2g - \rot_f(\partial_0S)$
from the Poincar\'e--Hopf theorem. Then the associated graded of 
the Turaev cobracket for 
$u = u_1\cdots u_l \in \widehat{T}(\gr H)$ with $u_i \in \gr H$ is given by 
\begin{eqnarray}
\grwt(\delta^f)(|u|) &=& 
\sum_{j<k}\langle u_j, u_k\rangle |u_{j+1}\cdots u_{k-1}|\wedge 
|u_{k+1}\cdots u_lu_1\cdots u_{j-1}|\label{eq:Tform}\\
&& + \sum_{j<k} |\mathfrak{z}(u_k, u_j)u_{j+1}\cdots u_{k-1}|\wedge
|u_{k+1}\cdots u_lu_1\cdots u_{j-1}|\nonumber\\
&& + \sum_{j<k} |\mathfrak{z}(u_j, u_k)u_{k+1}\cdots u_lu_1\cdots u_{j-1}|
\wedge |u_{j+1}\cdots u_{k-1}|\nonumber\\
&& + \sum^l_{i=1}c_f(u_i)\mathbf{1}\wedge |u_{i+1}\cdots u_lu_1\cdots u_{i-1}|.
\nonumber
\end{eqnarray}
Here $P\wedge Q \in \grwt(|\K\pi|)\otimes \grwt(|\K\pi|)$ means 
$P\otimes Q - Q\otimes P \in \grwt(|\K\pi|)\otimes \grwt(|\K\pi|)$ 
for $P, Q \in \grwt(|\K\pi|)$, and $\mathbf{1} = |1| \in \grwt(|\K\pi|)
= |\grwt(\K\pi)|$. 
If $c_f = 0$, then $\grwt(\delta^f)(|u|)$ equals Schedler's cobracket
\cite{Schedler}. 
In particular, the $\mathbf{1}\wedge \grwt(|\K\pi|)$-part of 
$\grwt(\delta^f)(|u|)$ equals 
\begin{equation}\label{eq:1Tform}
\mathbf{1}\}wedge\sum_{i=1}^l\Bigl|(\langle u_i, u_{i+1}\rangle - \mathfrak{z}(u_i, u_{i+1})
+ c_f(u_i)u_{i+1})u_{i+2}\cdots u_lu_1\cdots u_{i-1}
%u_{i+2}\cdots u_lu_1\cdots u_{i-1}
\Bigr|,
\end{equation}
where we agree $l+1:= 1$ and $l+2 := 2$.\par
Any special expansion gives the formal description of the Goldman bracket.
\begin{thm}[\cite{KK14, KK16, MT13, MTpre, Naef}]
\label{thm:Gform}
Let $\theta$ be a special expansion for $S = \Sigma_{g, n+1}$. Then 
$\theta: (|\widehat{\K\pi}|, [-,-]) \to (\grwt(|\K\pi|), [-,-]_{\grwt})$ 
is an isomorphism of Lie algebras. 
\end{thm}

From this theorem, any conjugate of a special expansion 
by a group-like element in $\mathrm{Gr}(\grwt(\K\pi))$ also 
gives the formal description of the Goldman bracket.
Recently Alekseev--Kawazumi--Kuno--Naef \cite{revisited} proved the converse.

\begin{thm}[\cite{revisited}]\label{thm:rev}
Let $\theta$ be a group-like expansion and assume that $\theta$ induces 
a Lie algebra isomorphism $\theta: (|\widehat{\K\pi}|, [-,-]) \to (\grwt(|\K\pi|), [-,-]_{\grwt})$. Then $\theta$ is conjugate to a special expansion 
by a group-like element in $\mathrm{Gr}(\grwt(\K\pi))$. 
\end{thm}

Lemma \ref{lem:inn}  and Theorem \ref{thm:center} 
are the keys to the proof of this theorem. 

\begin{lem}[{\cite[Proposition A.2]{genus0}} {\cite[Proposition 3.9]{revisited}}]\label{lem:inn}
\begin{enumerate}
\item Let $V$ be a finite-dimensional $\K$-vector space. 
Assume $x \in V\setminus\{0\}$ and $a \in \widehat{T}(V)$ satisfy 
$|ax^l| = 0 \in |\widehat{T}(V)|$ for any $l \geq 1$. 
Then we have $a \in [x, \widehat{T}(V)]$.
\item Let $V$ be a finite-dimensional symplectic $\K$-vector space 
with symplectic form $\omega_0 \in \wedge^2V$. 
Assume an element $a \in \widehat{T}(V)$ satisfies 
$|a{\omega_0}^l| = 0 \in |\widehat{T}(V)|$ for any $l \geq 1$. 
Then we have $a \in [\omega_0, \widehat{T}(V)]$.
\end{enumerate}
\end{lem}

\begin{thm}[\cite{CBEG07}]\label{thm:center}
The center of the Lie algebra $(\grwt(|\K\pi|), [-,-]_{\grwt})$ equals 
$$
Z(\grwt(|\K\pi|), [-,-]_{\grwt}) = |\K[[\omega]]|+ \sum^n_{j=0}
|\K[[z_j]]|.
$$
\end{thm}
This theorem can be proved in an elementary way using Lemma \ref{lem:inn} as in 
\cite[Theorem 5.4]{highergenus}.

\begin{rmk}\label{rmk:log}{\rm Let $S$ be the surface $\Sigma_{g,1}$ and $\theta$  a group-like expansion for $\pi_1(S)$.
Then, for any based loop $\gamma$, $\theta(\log\gamma) \in \mathrm{Lie}(\widehat{T}(H))$ 
is a Lie-like element. Since $|\mathrm{Lie}(\widehat{T}(H))| = |H|$, we have $|\theta(\log\gamma)| 
= |[\gamma]| \in |H|$, where $[\gamma] \in H$ is the homology class represented
by the loop $\gamma$. Hence, for any free loop $\alpha$, the isomorphism $|\theta|: 
|\K\pi|\to |\widehat{T}(H)|$ maps
$|\log\alpha | := 
\sum^\infty_{m=1}\frac{(-1)^{m+1}}{m}|(\alpha-1)^m| \in |\widehat{\K\pi}|$ 
to $|[\alpha]| \in |H| \subset |\widehat{T}(H)|$. 
}\end{rmk}
\par
In the succeeding sections, we need an action of the Goldman Lie algebra 
$|\Z\pi|$ on the group ring $\Z\pi$ of the fundamental group $\pi 
=\pi_1(S, *)$. 
A free loop $\alpha \in \hat\pi$ acts on a based loop $\gamma \in \pi$ 
as follows. Take generic immersions as their representatives.
Then the intersection points $\alpha\cap\gamma$ are finite, and 
we define $\sigma(\alpha)(\gamma) \in \Z\pi$ by 
$$
\sigma(\alpha)(\gamma) := \sum_{p\in \alpha\cap\gamma}
\varepsilon_p(\alpha, \gamma) \gamma_{*p}\alpha_p\gamma_{p*}
\in \Z\pi.
$$
Here $\varepsilon_p(\alpha, \gamma)\in \{\pm1\}$ is the local intersection 
number of $\alpha$ and $\gamma$ at the intersection point $p$, 
$\gamma_{*p}$ and $\gamma_{p*}$ are the segments of $\gamma$ 
from $*$ to $p$ and from $p$ to $*$, respectively, and
$\alpha_p \in \pi_1(S, p)$ the based loop along $\alpha$ with basepoint $p$.
Then it is well-defined, and defines a Lie algebra homomorphism 
$\Z\hat\pi/\Z\mathbf{1} \to \der(\Z\pi)$, where $\der(\Z\pi)$ 
is the Lie algebra of derivations of the group ring $\Z\pi$ \cite{KK14}.
We remark that the Goldman bracket of $\alpha \in \hat\pi$ and 
the free  loop $|\gamma| \in \hat\pi \subset |\Z\pi|$ is given by 
$[\alpha, |\gamma|] = |\sigma(\alpha)(\gamma)|$. The homomorphism
$\sigma$ is continuous with respect to the filtration $\{|(I\pi)^m|\} = \{|\K\pi(m)|\}$, and so 
descends to the Lie algebra homomorphism
\begin{equation}\label{eq:sigmag}
\sigma: |\widehat{\K\pi}|/\K\mathbf{1} \to \der(\widehat{\K\pi}).
\end{equation}
%for any field $\K$ with characteristic $0$. 
The right-hand side is the Lie algebra of {\it continuous} derivations of $\widehat{\K\pi}$. 
%Also the homomorphism $\sigma$ has the formal description by any special expansion. 
The associated graded of the homomorphism $\grwt(\sigma)$ is 
given by 
\begin{eqnarray}
\grwt(\sigma)(u)(v) &=& 
\sum_{i,j}\langle u_i, v_j\rangle
v_1\cdots v_{j-1}u_{i+1}\cdots u_lu_1\cdots u_{i-1}
v_{j+1}\cdots v_m\label{eq:Sform}\\
&& + \sum_{i,j}v_1\cdots v_{j-1}\mathfrak{z}(v_j, u_i)u_{i+1}\cdots u_lu_1\cdots u_{i-1}v_{j+1}\cdots v_m\nonumber\\
&& - \sum_{i,j}
v_1\cdots v_{j-1}u_{i+1}\cdots u_lu_1\cdots u_{i-1}\mathfrak{z}(u_i, v_j)v_{j+1}\cdots v_m.\nonumber
\end{eqnarray}
Here $u = u_1\cdots u_l$, $v = v_1\cdots v_m \in \widehat{T}(\gr H)$ 
for $u_i, v_j \in \gr H$ as above.
As is proved in \cite{KK14, KK16, MT13, MTpre, Naef}, any special expansion 
induces a formal description of the homomorphism $\sigma$.\par
A free loop $\alpha$ also acts on the free $\Z$-module over 
the homotopy set of paths $[([0,1], 0, 1), (S, *_j, *_k)]$, $0 \leq j, k \leq n$.
We denote the action by  
$\sigma: |\Z\pi|/\Z\mathbf{1} \to \mathrm{End}(\Z[([0,1], 0, 1), (S, *_j, *_k)])$. 
Using the path $\ell_j$ in \S\ref{subsec:exp}, we define $\sigma_j(\alpha) \in \Z\pi$
by $\sigma_j(\alpha) = (\sigma(\alpha)(\ell_j)){\ell_j}\inv \in \Z\pi$. 
Then we have $\sigma(\alpha)({\ell_j}\inv) = - {\ell_j}\inv \sigma_j(\alpha)$ and so $\sigma_j([\alpha, \beta]) 
= \sigma(\alpha)(\sigma_j(\beta)) 
-  \sigma(\beta)(\sigma_j(\alpha)) 
+  \sigma_j(\alpha)\sigma_j(\beta)
-  \sigma_j(\beta)\sigma_j(\alpha)
$
for any $\alpha, \beta \in \hat\pi$. Hence the homomorphism $\sigma$ 
lifts to a homomorphism to the semi-direct product
$$
\hat\sigma:= (\sigma, \sigma_1, \dots, \sigma_n): 
|\Z\pi|/\Z\mathbf{1} \to \der(\Z\pi)\ltimes (\Z\pi)^{\oplus n},
$$
which also has a formal description by any special expansion 
\cite{KK16}. The associated graded of the homomorphism
$\hat\sigma$ 
$$
\aligned
\grwt(\hat\sigma) = 
& (\grwt(\sigma), \grwt(\sigma_1), \dots, \grwt(\sigma_n)): \\
& \grwt(|\widehat{\K\pi}|) \to \der(\grwt(\K\pi))\ltimes (\grwt(\K\pi))^{\oplus n}
\endaligned
$$
has its image in a subalgebra $\tDer$ defined by 
$$
\tDer := \{(u, u_1, \dots, u_n) \in 
\der(\grwt(\K\pi))\ltimes (\grwt(\K\pi))^{\oplus n}; \,\,
u(z_j) = [z_j, u_j], 1 \leq j \leq n\}.
$$
\par
The homotopy intersection form introduced independently 
by Papakyriakopoulos \cite{Papa} and Turaev \cite{Tu78} is an upgrade of the homomorphism $\sigma$.
Massuyeau and Turaev \cite{MT13, MTpre} gave an explicit 
tensorial description of the homotopy intersection form by any special expansion.
%

%\section{Application to the mapping class group}
%\label{sec:mcg}

\subsection{The logarithms of mapping classes}
\label{subsec:log}

In this subsection we confine ourselves to the case $S = \Sigma_{g,1}$.
See \cite{KK15, KK16} for general compact surfaces $S$. 
%\par
%Choose a basepoint $* \in \partial S$, and write simply $\pi := 
%\pi_1(S, *)$. 
Instead of $\gamma_0$, 
we denote by $\zeta \in \pi = \pi_1(S, *)$ the negative boundary loop 
with basepoint $* \in \partial S$. 
Then the image $\sigma(|\Z\pi|/\Z\mathbf{1})$ is included in 
$\der_\zeta(\Z\pi)$, where $\der_\zeta(\Z\pi)$ 
is the Lie algebra of derivations of the group ring $\Z\pi$ annihilating 
the boundary loop $\zeta \in \pi$. Hence we have 
the Lie algebra homomorphism
\begin{equation}\label{eq:sigma}
\sigma: |\widehat{\K\pi}|/\K\mathbf{1} \to \der_\zeta(\widehat{\K\pi}).
\end{equation}
%for any field $\K$ with characteristic $0$. 
The right-hand side is the Lie subalgebra of $\der(\widehat{\K\pi})$ 
annihilating the boundary loop $\zeta$. Using the formal descriptions
of the Goldman bracket and the homomorphism $\sigma$, 
one can prove the homomorphism $\sigma$ in 
\eqref{eq:sigma} is an isomorphism of Lie algebras. For details, see \cite{KK16}.
We denote by $|\widehat{\K\pi}|^+$ the closure of $|(I\pi)^3|$
in $|\widehat{\K\pi}|$. 
From \eqref{eq:Gfil}, $|\widehat{\K\pi}|^+$ is a pro-nilpotent
Lie subalgebra of  $|\widehat{\K\pi}|$. 
In particular, the Baker--Campbell--Hausdorff series defines a group structure 
on the subalgebra $|\widehat{\K\pi}|^+$
\par
The mapping class group $\mathcal{M}(S)$, 
the path-component group of the group of 
orientation-preserving diffeomorphisms 
fixing the boundary $\partial S$ pointwise, 
acts on $\widehat{\K\pi}$ in an obvious way. 
We consider the subset 
$$
\mathcal{M}^\circ := \left\{\varphi \in \mathcal{M}(S); \,\,
\text{the logarithm $\log\varphi
:= \sum^\infty_{m=1}\frac{(-1)^{m+1}}{m}(\varphi-1)^m$ on $\widehat{\K\pi}$ exists}\right\}.
$$
It is a proper subset of $\mathcal{M}(S)$, since any mapping class
whose action on the homology group $H$ is semi-simple 
can not be there.
But it contains all the Dehn twists and 
the Johnson group $\mathcal{K}_{g,1} := \ker(\mathcal{M}(\Sigma_{g,1})
\to \aut(\K\pi/(I\pi)^3))$. 
Then we can define a map 
$$
\tau: \mathcal{M}^\circ
\overset{\log}\to \der_\zeta(\widehat{\K\pi}) \overset{\sigma\inv}\to 
|\widehat{\K\pi}|/\K \mathbf{1}, \quad
\varphi \mapsto \sigma\inv(\log\varphi),
$$
which is a geometric extension of the Johnson homomorphisms \cite{Joh83}.
As was proved in \cite{KK14}, we have $\tau(t_C) = L(C) := \frac12(\log C)^2 
\in |\widehat{\K\pi}|/\K \mathbf{1}$ for any simple closed curve $C \subset \IntS$.
This formula is generalized to any compact 
oriented surface by \cite{KKpre} and \cite{MT13}. In particular, 
$\tau(t_C)$ is an infinite sequence of the powers $(C-1)^m$ of $C-1$, $m \geq 0$.
On the other hand, we have 
$\tau(\mathcal{K}_{g,1}) \subset |\widehat{\K\pi}|^+/\K\mathbf{1}$, and 
$\tau: \mathcal{K}_{g,1} \to |\widehat{\K\pi}|^+/\K\mathbf{1}$ is a group homomorphism.
For details, see \cite{KK16}. \par
\begin{rmk}{\rm 
Following Turaev's quantization of the Goldman--Turaev Lie bialgebra \cite{Tu91}, 
S.\ Tsuji `skein-ized' the logarithms of mapping classes in \cite{Tsu1, Tsu2, Tsu4}. 
His constructions yield invariants of integral homology $3$-spheres
\cite{Tsu3, Tsu4}.
Very recently he proved that the invariants recover the
universal quantum $sl(n)$ invariants \cite{Tsu5}.
}
\end{rmk}

%\subsection{A constraint on the Johnson image}
%\label{subsec:Johnson}

Now we review a relation between the map $\tau$ and the Turaev cobracket. 
In \cite{KK15} it is proved that $\delta\circ\tau = 0: \mathcal{M}^\circ \to 
(|\widehat{\K\pi}|/\K\mathbf{1})^{\otimes 2}$ by using a based variant 
$\mu$ of the 
Turaev cobracket $\mu$ \cite{Tu78}. 
Similarly one can prove $\delta^f\tau(\varphi) = 0$
for any $\varphi \in \mathcal{M}^\circ$ satisfying $f\varphi = f$. 
We remark that any $\varphi \in \mathcal{K}_{g,1}$ acts on the 
set $F(S)$ trivially. 
Here we give an alternative proof for $\delta^f\tau(\mathcal{K}_{g,1}) = 0$. 
\par
\begin{lem}\label{lem:vanishdlt} 
For any simple closed curve $C$ in $M = 
\IntS$, we have 
$$
\delta^f L(C) = (\rot_f C)(\mathbf{1} \otimes \log C - \log C\otimes \mathbf{1})
\in |\widehat{\K\pi}|^{\otimes 2}.
$$
\end{lem}
\begin{proof} From Lemma \ref{lem:iterate}, we have 
$$
\delta^f(C^m) = \delta(C^m) + \rot_f(C^m)(\mathbf{1}\otimes C^m-C^m\otimes \mathbf{1})
= m\cdot\rot_f(C)(\mathbf{1}\otimes C^m-C^m\otimes \mathbf{1})
$$
for any $m \geq 0$. Hence, for any formal power series $h(z)$ in a variable $z-1$, 
we have 
$$
\delta^f(h(C)) = \rot_f(C)(\mathbf{1}\otimes Cf'(C)-Cf'(C)\otimes \mathbf{1}).
$$
The lemma follows from the fact that $z\frac{d}{dz}(\frac12(\log z)^2) = \log z$. 
\end{proof}

As was stated in Remark \ref{rmk:log}, the logarithm $\log C = \sum^\infty_{m=1}\frac{(-1)^{m+1}}{m} | (C-1)^m| \in |\widehat{\K\pi}|$ depends only on the homology class $[C] \in H_1(M; \K)$.
In particular, we have $\delta^f L(C) = 0$ if $[C] = 0 \in H_1(M; \K)$. 
\begin{thm}[\cite{KK15}]\label{thm:delta}
$(\delta^f\circ \tau)(\mathcal{K}_{g,1}) = 0$.
\end{thm}
\begin{proof} As was proved by Johnson \cite{J85}, the group $\mathcal{K}_{g,1}$ 
is generated by Dehn twists along simple closed curves 
whose homology classes vanish. The value of the composite $\delta^f\circ \tau$ vanishes at any of such 
Dehn twists by Lemma \ref{lem:vanishdlt}. 
Lemma \ref{lem:cocycle} implies that the kernel $\ker\delta^f$ is 
a Lie subalgebra. 
The theorem follows from the fact that 
$\tau$ is a group homomorphism.
\end{proof}
\par
Here we briefly review the Johnson homomorphisms \cite{Joh83}.
Let $\{\Gamma_k\pi\}_{k=1}^\infty$ be the lower central series of 
the fundamental group $\pi = \pi_1(S, *)$. This series is defined by 
$\Gamma_1\pi = \pi$ and $\Gamma_{k+1}\pi = [\Gamma_k\pi, \pi]$
for $k \geq 1$. The Andreadakis--Johnson filtration $\{\mathcal{M}(k)\}_{k=0}^\infty$ 
of the mapping class group $\mathcal{M}(S) = \mathcal{M}(\Sigma_{g,1})$ 
is defined by $\mathcal{M}(k) := \ker(\mathcal{M}(S) \to \aut(\pi/\Gamma_{k+1}\pi))$ for $k \geq 0$. We have $\mathcal{M}(0) =  \mathcal{M}(S)$, 
$\mathcal{M}(2) = \mathcal{K}_{g,1}$ and $\mathcal{M}(1)$ is called 
the Torelli group. One can check that $\tau(\mathcal{M}(k))$ is included 
in the closure of $|(I\pi)^{k+2}|$ in $|\widehat{\K\pi}|/\K \mathbf{1}$ for any $k \geq 1$.
Hence we obtain the associated graded
$$
\gr(\tau): \prod^\infty_{k=1} \mathcal{M}(k)/\mathcal{M}(k+1) 
\to \grwt(|\K\pi|),
$$
which equals the Johnson homomorphisms \cite{KK16}, and is a Lie algebra homomorphism \cite{MoAQ}. From the facts that $\delta\circ\tau = 0$
and $\delta^f\circ\tau = 0$, we can 
derive some constraints on the image of the Johnson homomorphism $\gr(\tau)$. 
As was computed in \cite[\S6.4]{KK15}, the original Turaev cobracket 
$\delta$ includes the Morita trace \cite{MoAQ}.
But, as was proved by Enomoto--Kuno--Satoh \cite{EKS18}, 
the constraint coming from the original $\gr(\delta)$ 
does not include the Enomoto--Satoh trace \cite{ES}, a refinement of the Morita trace.
The Enomoto--Satoh trace equals 
the $\mathbf{1}\wedge \grwt(|\K\pi|)$-part of 
$\grwt(\delta^f)(u)$ \eqref{eq:1Tform} for $S = \Sigma_{g,1}$
\cite{Ka15}.
Thus we need to consider the formality problem of the framed Turaev cobracket 
$\delta^f$. Moreover, if $g = 0$ and $c_f = 0$, 
then the $\mathbf{1}\wedge \grwt(|\K\pi|)$-part equals 
minus the Alekseev--Torossian divergence cocycle $\div$ in the Kashiwara--Vergne 
problem \cite{AN17}. This leads us to study the Kashiwara--Vergne problem.
\par

\section{Kashiwara--Vergne problems and the formality of the Turaev cobracket}
\label{sec:form}

As will be shown in this section, 
the Kashiwara--Vergne problem can be regarded
as a problem on a group-like expansion for 
a framed surface $(\Sigma_{g, n+1}, f)$. 
But, in the formulation in \cite{AT12}, 
it is a problem for a (continuous) automorphism
of the Lie algebra $\mathrm{Lie}(\grwt(\K\pi))$
with some auxiliary data. \par
We write simply 
$A := \grwt(\K\pi)$ and
$L := \mathrm{Lie}(\grwt(\K\pi))$.
The Lie algebra $L$ 
has a natural filtration coming from the weight filtration
and a group structure by the Baker--Campbell--Hausdorff series, 
and we denote by $\aut^+(L)$
the group of filtration-preserving automorphisms
of $L$ whose associated graded is the identity. 
Then one can introduce a subgroup $\taut = \taut^\gn$ 
of the semi-direct product of $\aut^+(L)$ and $L^{\oplus n}$
defined by 
$$
\taut := \{(F, f_1, \cdots, f_n) \in \aut^+(L)\ltimes L^{\oplus n}; \,\,
F(z_j) = \exp(-f_j)z_j\exp(f_j),\,\, 1 \leq j \leq n\}.
$$
We denote by $\overline{\taut}$ the image of the $\taut$ under the first 
projection $\taut \hookrightarrow \aut^+(L)\ltimes L^{\oplus n} \overset{\mathrm{pr}_1}\to \aut^+(L)$. \par
The Lie algebra of the group $\aut^+(L)$ is the Lie algebra 
$\der^+(L)$ of 
derivations of $L$ of positive degree with respect to the weight filtration, 
and the Lie algebra of $\taut$ coincides with a Lie subalgebra 
$\tder^+$ of 
the semi-direct product of $\der^+$ and $L^{\oplus n}$ given by 
$$
\tder^+ = \{(u, u_1, \cdots, u_n) \in \der^+(L)\ltimes L^{\oplus n}; \,\,
u(z_j) = [z_j, u_j],\,\, 1 \leq j \leq n\},
$$
which can be regarded as a Lie subalgebra of $\tDer$ in an obvious way.
In the sequel we write simply $u$ for $(u, u_1, \cdots, u_n) \in \tder^+$
for simplicity.
Here we remark that the exponential $\exp: \tder^+ \to \taut$ is a bijection, 
and that any graded Lie algebra action of $\tder^+$ 
on a graded module $M$ integrates to a group action of the group 
$\taut$ on $M$. Moreover any graded Lie algebra $1$-cocycle 
$c: \tder^+ \to M$ integrates to a group cocycle $C: \taut \to M$ 
satisfying the condition $\frac{d}{dt}C(\exp(tu))\vert_{t=0} = c(u)$
for any $u \in \tder^+$. The group cocycle $C$ is explicitly written as
\begin{equation}\label{eq:integ}
C(\exp(u)) = \frac{e^u-1}{u}\cdot c(u) 
\left(= \sum^\infty_{m=0}\frac{1}{(m+1)!}u^m (c(u))\right)
\end{equation}
for any $u \in \tder^+$. 
\par
Now we recall the free generator system 
$\{\alpha_i, \beta_i, (1 \leq i \leq g), \,\, \gamma_j (1 \leq j \leq n)\}$
of the group $\pi = \pi_1(S, *)$, and the exponential expansion 
$\theta_{\exp}: \widehat{\K\pi} \to \grwt(\K\pi)$ associated with the 
generator system introduced in \S\ref{subsec:exp}. Then any tangential expansion $\theta$ equals 
$F\inv\circ\theta_{\exp}$ for some unique automorphism $F \in \overline{\taut}$.
In other words, the set of tangential expansions is a torsor of 
the group $\overline{\taut}$. 
From the definition of a special expansion, 
$F\inv\circ\theta_{\exp}$ is special if and only if $F$ satisfies the condition
$$
F\left(\sum^g_{i=1}[\bar x_i, \bar y_i]+
\sum^n_{j=1}z_j\right) = \log\left(\prod^g_{i=1}(e^{\bar x_i}e^{\bar y_i}e^{-\bar x_i}e^{-\bar y_i})
\prod^n_{j=1}e^{z_j}\right),
\leqno{\rm{(KVI)}}
$$
which is exactly the first condition of the Kashiwara--Vergne problem.
Hence we may assume that an automorphism $F \in \taut$ satisfies the condition (KVI).
In this section, under the assumption, 
we will provide a necessary and sufficient condition for $F$ 
to induce the formality of the Turaev cobracket.
This is exactly the same as the other condition of the Kashiwara--Vergne problem, that is, (KVII).

\subsection{Formality in genus $0$}
\label{subsec:zero}

A formal description of the framed cobracket $\delta^f$ in genus $0$ 
was given first by Massuyeau \cite{Mas15} by using the Kontsevich integral.
After the work \cite{genus0} we review here, 
Alekseev and Naef \cite{AN17} gave it by using 
the Knizhnik--Zamolodchikov connections. 
This seems the simplest proof of the formality of the Goldman--Turaev 
Lie bialgebra in genus $0$. \par
Now we recall the Kashiwara--Vergne problem \cite{KV78}
in the formulation of Alekseev--Torossian \cite{AT12}.
This corresponds to the genus $0$ surface $\Sigma_{0,n+1}$, $n\geq 2$, with the unique framing $f$ satisfying $\rot_f(\partial_jS) = -1$
for any $j \geq 1$, which is realized by a standard embedding of 
$\Sigma_{0, n+1}$ into the unit disk in $\R^2$ 
mapping $\partial_0S$ to the unit circle.
Then we have $c_f = 0$, $H = H^{(2)}$, and so $\grwt(\K\pi) = \widehat{T}(H)$.
In particular, $L = \mathrm{Lie}(\grwt(\K\pi))$ equals 
the complete free Lie algebra generated by $H = \oplus^n_{j=1}\K z_j$. 
Any element $a \in A = \widehat{T}(H)$ is uniquely written as $a = a^0 + 
\sum^n_{i=1}a^iz_i$ with $a^0 \in \K$ and $a^i \in A$. 
Then the Alekseev--Torossian divergence cocycle $\div: \tder^+ \to |A|
= |\widehat{T}(H)|$ is defined by
$$
\div(u) := \sum^n_{i=1}|z_i (u_i)^i| \in |A|
$$
for $u = (u, u_1, \dots, u_n) \in \tder^+$ \cite{AT12}.
We remark the divergence $\div$ naturally extends to the Lie algebra $\tDer$.
The Jacobian cocycle $j: \taut \to |A|$ is the integration of 
the divergence cocycle $\div$ by the construction \eqref{eq:integ}. 
\begin{dfn}[\cite{AT12}]\label{dfn:KV0}
An element $F \in \taut$ is a solution to the Kashiwara--Vergne problem
(of type $(0, n+1)$) if it satisfies the conditions {\rm (KVI)} and 
$$
j(F\inv) = \bigl|\sum^n_{j=1}h_j(z_j) - h\left(\sum^n_{i=1}z_i\right) \bigr| 
\leqno{\rm{(KVII)}}
$$
for some $h(s)$ and $h_j(s) \in \K[[s]]$, $1 \leq j \leq n$.
\end{dfn}
From Theorem \ref{thm:center} the condition (KVII) is equivalent to 
the condition 
$$
j(F\inv) \in Z(\grwt(|\K\pi|), [-,-]_{\grwt}).
\leqno{\rm{(KVII')}}
$$
The functions $h_j$ and $h$ are the same modulo their linear parts
\cite[Theorem 8.7]{genus0} and called the {\it Duflo function}. 
The existence of solutions to the problem is proved by 
Alekseev--Meinrenken \cite{AM06} and Alekseev--Torossian \cite{AT12}. 
\par
The main theorem in this subsection is
\begin{thm}[\cite{genus0}]\label{thm:genus0}
Let $F \in \taut$ satisfy the condition {\rm (KVI)}. 
Then the special expansion $F\inv\circ\theta_{\exp}$ 
induces the formal description of the framed Turaev cobracket, 
or equivalently, induces an isomorphism of Lie bialgebras
$$
F\inv\circ\theta_{\exp}: (|\widehat{\K\pi}|, [-,-], \delta^f) \to (\grwt(|\K\pi|), [-,-]_{\grwt}, 
\grwt(\delta^f)),
$$
if and only if $F$ is a solution to the Kashiwara--Vergne 
problem. 
\end{thm}

Here recall the Lie algebra homomorphism
$\grwt(\hat\sigma) = (\grwt(\sigma), \grwt(\sigma_1), \dots, \grwt(\sigma_n)): 
\grwt|\widehat{\K\pi}| \to \tDer$ introduced in \ref{subsec:gold}. 
In the present case $S= \Sigma_{0,n+1}$, for any element $u = z_{k_1}\cdots z_{k_m} \in A$, $1 \leq k_l \leq n$,
we have 
$\grwt(\sigma_j)(|u|) = \sum_l \delta_{jk_l}z_{k_{l+1}}\cdots z_{k_m}z_{k_1}\cdots z_{k_{l-1}}$, and so 
the $\mathbf{1}\wedge \grwt(|\K\pi|)$-part of 
$\grwt(\delta^f)(|u|)$ equals $-\mathbf{1}\wedge \div(\grwt(\sigma)(|u|))$. 
A ``double version" of this observation is one of the main ingredients
in our approach 
to the formality of the Turaev cobracket. So we introduce 
the {\it double divergence} $\mathrm{tDiv}: \tDer \to |A|\otimes |A|$ 
which maps $(u, u_1, \dots, u_n) \in \tDer$ to 
$$
\mathrm{tDiv}(u) := \sum_{j=1}^n\bigl|(z_j\otimes 1)(\frac{\partial}{\partial z_j}u_j)
- (\frac{\partial}{\partial z_j}u_j)(1\otimes z_j)\bigr|.
$$
Here $|\cdot|: A\otimes A 
\to |A|\otimes |A|$, $a_1\otimes a_2 \mapsto |a_1|\otimes |a_2|$, 
is the quotient map, and $\frac{\partial}{\partial z_j}: A \to A\otimes A$ 
is defined by $\frac{\partial}{\partial z_j}(z_{k_1}\cdots z_{k_m}) 
= \sum_l\delta_{jk_l} z_{k_1}\cdots z_{k_{l-1}}\otimes z_{k_{l+1}}\cdots z_{k_m}$,
$1 \leq k_l\leq n$. The divergences $\div$ and $\mathrm{tDiv}$ 
satisfiy the commutative diagram
\begin{equation}\label{eq:CD}
\begin{CD}
\tder^+ @>\div>> |A|\\
@V{\text{incl}}VV @V{\tilde\Delta}VV\\
\tDer @>{\mathrm{tDiv}}>> |A|\otimes|A|,
\end{CD}
\end{equation}
where $\tilde\Delta = (1\otimes \mathrm{antipode})\circ \Delta:
|A| \to |A|\otimes |A|$. We remark $\tder^+$ in the diagram 
\eqref{eq:CD} cannot be replaced by $\tDer$.
The cocycle $\mathrm{tDiv}: \tder^+ \to |A|\otimes |A|$ is integrated 
to a group cocycle $J: \taut \to |A|\otimes|A|$. 
The diagram \eqref{eq:CD} implies $J = \tilde\Delta\circ j: 
\taut \to |A| \to |A|\otimes |A|$. 
\par
Now we take the exponential expansion $\theta_{\exp}: 
\widehat{\K\pi} \overset\cong\to \grwt{\K\pi} = A$, and 
the conjugates of
$\delta^f_{\exp}: |A| 
\to |A| \otimes |A|$ 
%,$|u| \mapsto (\theta\otimes\theta)\delta^f\theta\inv(|u|)$, 
and $\hat\sigma_{\exp}: |A| \to \tDer$
by the isomorphism $\theta_{\exp}$.
Then we have the following.
\begin{thm}\label{thm:exp}
$$
\delta^f_{\exp} = \mathrm{tDiv}\circ \hat\sigma_{\exp}: |A| \to |A|\otimes |A|.
$$
\end{thm}
The theorem is the key to the proof of the formality of the framed
Turaev cobracket. Its proof is straightforward and long, and involves 
the homotopy intersection form \cite{Papa, Tu78, MT13}
in the framework of non-commutative geometry in the sense of 
van den Bergh \cite{vdB}. 
It was Massuyeau and Turaev \cite{MT14} who first 
reformulated the homotopy intersection form 
in the framework of van den Bergh.
Very recently, Alekseev, Kuno, Naef and the author \cite{double} obtained 
a direct topological proof of Theorem \ref{thm:exp}
by using a double version of Turaev's gate derivatives
\cite{Tu19, Tu19b}.  
As a corollary of Theorem \ref{thm:exp} we have 
$\grwt(\delta^f) = \mathrm{tDiv}\circ \grwt(\hat\sigma)$. 
\par
\begin{proof}[Proof of Theorem \ref{thm:genus0}]
Let $F$ be an element of $\taut$ such that $\theta := F\inv\circ\theta_{\exp}$ 
is a special expansion. The conjugate $F^*\delta^f_{\exp}$ by $F$ equals
the induced cobracket $\delta^f_\theta$ on $|A| $ 
by the isomorphism $\theta$. 
Then we have $F^*\delta^f_{\exp} = (F^*\mathrm{tDiv})\circ
(F^*\hat\sigma_{\exp})$. Any special expansion induces 
the formal description of $\hat\sigma$, i.e., 
$F^*\hat\sigma_{\exp} = \grwt(\hat\sigma)$. On the other hand,
if we denote the Lie algebra coboundary operator by $\partial$, i.e., 
$(\partial w)(u) = u\cdot w$ for 
$u \in \tDer$ and 
$w \in |A|\otimes |A|$, 
one can prove
$F^*\mathrm{tDiv}(u) = \mathrm{tDiv}(u) + u\cdot J(F\inv)
= (\mathrm{tDiv} + \partial J(F\inv))(u)
= (\mathrm{tDiv} + \partial\tilde\Delta j(F\inv))(u)
$
\cite[Proposition 4.9]{highergenus}. 
Hence we have 
$$
\delta^f_\theta = F^*\delta^f_{\exp} =  \grwt{\delta^f} + 
(\partial \tilde\Delta j(F\inv))\circ\grwt(\hat\sigma).
$$
In particular, if $j(F\inv)$ is in the center of the bracket
$[-,-]_{\grwt}$, we obtain $\delta^f_\theta = \grwt{\delta^f}$.
Conversely assume $\delta^f_\theta = \grwt{\delta^f}$.
Looking at the $\mathbf{1}\otimes |A|$-part, we find out that 
$|\grwt(\sigma)(j(F\inv))(w)| 
= [j(F\inv), |w|]_{\grwt} = 0$ for any $w \in A$ of degree $\geq 2$. 
From Lemma \ref{lem:inn}, the derivation $\grwt(\sigma)(j(F\inv))$ 
is inner, and so $j(F\inv) \in |A|$ is in the center of the bracket
$[-,-]_{\grwt}$. This proves Theorem \ref{thm:genus0}.
\end{proof}

\subsection{Higher genus Kashiwara--Vergne problems}
\label{subsec:HKV}

Theorem \ref{thm:genus0} provides a necessary and sufficient condition 
for a special expansion to induce the formal description of the Turaev 
cobracket in genus $0$. This leads us to generalize the Kashiwara--Vergne problem 
to positive genus surfaces: If an element $F \in \taut^\gn$ satisfies
the condition (KVI), and the expansion $F\inv\circ\theta_{\exp}$ 
induces the formal description of the Turaev cobracket, 
then it should be a solution to the higher genus Kashiwara--Vergne problem.
\par
For this purpose, we need to modify the divergence. 
In fact, the topological nature of the $\alpha_i$'s and $\beta_i$'s 
are quite different from that of $\gamma_j$'s, 
which brings us the function 
$$
r(s) := \log((e^s-1)/s).
$$ 
Moreover we consider all possible framings on the surface $S = \Sigma_{g, n+1}$.
Recall that we have fixed the generating system $\{\alpha_i, \beta_i, \gamma_j\}$
of the fundamental group $\pi = \pi_1(S, *)$.
There is a unique framing $f^{\mathrm{adp}}$ on $S$ 
such that $\rot_{f^{\mathrm{adp}}}(|\alpha_i|) = 
\rot_{f^{\mathrm{adp}}}(|\beta_i|) = 0$ for any $1 \leq i \leq g$ 
and $\rot_{f^{\mathrm{adp}}}(\partial_jS) = -1$ for any 
$1 \leq j \leq n$. For any framing $f$ on $S$, there exists a unique
$\chi \in H^1(S; \Z)$ such that $\rot_f(\alpha) 
= \rot_{f^{\mathrm{adp}}}(\alpha)
+ \langle \chi, [\alpha]\rangle$ for any immersed loop $\alpha$. 
We introduce $q \in \hom(H^{(2)}, \Z)$ and $p \in H/H^{(2)} 
\subset \gr H$ by $q := c_f\vert_{H^{(2)}} = \chi\vert_{H^{(2)}}$
and $\langle p, \cdot\rangle = \chi-c_f : H/H^{(2)} \to \Z$. 
Here we remark that the intersection number on $H/H^{(2)}$ 
is non-degenerate. 
The map $c_q: \tDer \to |A|$, $(u, u_1, \dots, u_n) 
\mapsto \sum_{j=1}^n q(z_j)|u_j|$ restricts to a cocycle 
on $\tder^+$, and integrates to 
a group cocycle $C_q: \taut\to |A|$.
%\par
In this situation the double divergence 
$\mathrm{tDiv}: \tDer \to |A|\otimes|A|$ maps
$(u, u_1, \dots, u_n)$ to 
$$
\bigl|\sum_i(\frac{\partial}{\partial x_i}u(x_i) + 
\frac{\partial}{\partial y_i}u(y_i)) + 
\sum^n_{j=1} ((z_j\otimes 1)(\frac{\partial}{\partial z_j}u_j)
- (\frac{\partial}{\partial z_j}u_j)(1\otimes z_j)
\bigr|.
$$
This is a cocycle and integrates to a group cocycle $J: \taut 
\to |A|\otimes |A|$. In a similar way to the genus $0$ case, 
there exists a unique group 
cocycle $j: \taut \to |A|$ such that $J = \tilde\Delta \circ j$. 
We denote $j_q := j - C_q: \taut \to |A|$, and 
$\mathbf{r} := \sum_i|r(x_i) + r(y_i)|$ for $r(s) = \log((e^s-1)/s)$.

\begin{dfn}[\cite{announce, highergenus}]\label{dfn:KV+}
An element $F \in \taut$ is a solution to the Kashiwara--Vergne 
for the framed surface $(\Sigma_{g, n+1}, f)$ if it satisfies the condition (KVI)
and the second condition 
$$
j_q(F) = \mathbf{r} + |p| + \sum^n_{j=1}|h_j(z_j)| 
- |h(\xi)|
\leqno{\rm{(KVII)}}
$$
for some $h(s)$ and $h_j(s) \in \K[[s]]$, $1 \leq j\leq n$.
Here $\xi := \log\left(\prod^g_{i=1}(e^{x_i}e^{y_i}e^{-x_i}e^{-y_i})
\prod^n_{j=1}e^{z_j}\right) \in L \subset A$ is the right-hand side of the condition {\rm(KVI)}. 
\end{dfn}

In the case $g >0$ we do not know if there is a non-trivial relation 
among the functions $h$ and $h_j$'s or not. 
By Theorem \ref{thm:center}, 
the second condition is equivalent to the condition
$$
j_q(F\inv) + F\inv(\mathbf{r}+|p|) \in  Z(\grwt(|\K\pi|), [-,-]_{\grwt}).
\leqno{\rm{(KVII')}}
$$
\par
With a few exceptions there exist solutions to the problem.
\begin{thm}[{\cite[Theorem 6.1]{highergenus}}]\label{thm:exist}
\begin{enumerate}
\item If $g \geq 2$, then the Kashiwara--Vergne problem has a solution 
for any framing.
\item Assume $g = 1$. Then the Kashiwara--Vergne problem 
has a solution if and only if $q\neq 0$ or $q=p=0$. 
\end{enumerate}
\end{thm}
We construct solutions by gluing solutions in low genera.
In particular, our construction in the case $g=1$ follows 
Enriquez' construction of elliptic associators \cite{Enriquez}.
\par
Recently Hain \cite{Hain18} proved that the framed cobracket 
is a morphism of mixed Hodge structures. Combining his result
and Theorems \ref{thm:genus0} and \ref{thm:genus+}, one can 
deduce the existence of solutions to the Kashiwara--Vergne 
problem except in the genus $1$ case with $q=0$ and $p\neq 0$
\cite{Hain18}.
\par

The problem is related to the formality of the Turaev cobracket,
as desired.
\begin{thm}[\cite{highergenus}]\label{thm:genus+}
Let $F \in \taut$ satisfy the condition (KVI). 
Then the special expansion $F\inv\circ\theta_{\exp}$ 
induces the formal description of the framed Turaev cobracket $\delta^f$, 
or equivalently, induces an isomorphism of Lie bialgebras 
$$
F\inv\circ\theta_{\exp}: (|\widehat{\K\pi}|, [-,-], \delta^f) \to (\grwt(|\K\pi|), [-,-]_{\grwt}, 
\grwt(\delta^f)),
$$
if and only if $F$ is a solution to the Kashiwara--Vergne 
problem for the framed surface. 
\end{thm}

An analog of Theorem \ref{thm:exp} is the key to the proof 
of the theorem. We define $\gdiv^f: \tDer \to |A|\otimes |A|$ 
by $\gdiv^f := \mathrm{tDiv}+ \partial(\tilde\Delta\mathbf{r}
+ |p|\wedge \mathbf{1}) - c_q\wedge \mathbf{1}$. Here $\partial$ means 
the Lie algebra coboundary operator. 
\begin{thm}[{\cite[Corollary 4.8]{highergenus}}]
$$
\delta^f_{\exp} = \gdiv^f\circ\hat\sigma_{\exp}: 
|A| \to |A|\otimes |A|.
$$
\end{thm}
In a similar way to the genus $0$ case,
we can prove that $F^*\delta^f_{\exp}= \grwt(\delta^f)$ 
if and only if $j_q(F\inv)+ F\inv(\mathbf{r}+|p|) \in |A|$ is 
in the center for the bracket $[-,-]_{\grwt}$. 
This proves Theorem \ref{thm:genus+}.

%%%%%%%%%%%

%The double divergence $\ddiv$ is expressed
%topologically in terms of 
%a double version of Turaev's gate derivative \cite{Tu19}.
%It comes from a handle decomposition of the surface $S$. 
%We will write the details elsewhere.

%\subsection{Topological divergence}
%\label{subsec:topdiv}
%
%\par
%\bigskip
%{\Huge FIGURE OF GATES}\par
%Figure 3\par
%\bigskip

\end{document}